\documentclass[10. pt,a4paper]{amsart}
\usepackage{rotating}

\usepackage[utf8]{inputenc}
\usepackage[all]{xy}
\usepackage{amsmath}
\usepackage{amscd,amssymb,amsopn,amsmath,amsthm,graphics,amsfonts,enumerate,verbatim,calc}
\headsep=1cm \numberwithin{equation}{section}
\newtheorem{theorem}{Theorem}[section]
\newtheorem{lemma}[theorem]{Lemma}
\newtheorem{proposition}[theorem]{Proposition}
\newtheorem{corollary}[theorem]{Corollary}
\newtheorem{notation}[theorem]{Notation}
\newtheorem{observation}[theorem]{Observation}
\theoremstyle{definition}

\newtheorem{definition}[theorem]{Definition}\newcommand{\up}[1]{{{}^{#1}\!}}
\theoremstyle{remark}
\newtheorem{remark}[theorem]{Remark}\newtheorem{fact}[theorem]{Fact}
\newtheorem{example}[theorem]{Example}
\newtheorem{question}[theorem]{Question}

\newtheorem{problem}[theorem]{Problem}
\newtheorem{acknowledgement}{Acknowledgement}

\newcommand{\Ass}{\operatorname{Ass}}
\newcommand{\F}{\operatorname{F}}

\newcommand{\Syz}{\operatorname{Syz}}

\newcommand{\Assh}{\operatorname{Assh}}
\newcommand{\Spec}{\operatorname{Spec}}

\newcommand{\Tr}{\operatorname{Tr}}

\newcommand{\End}{\operatorname{End}}

\newcommand{\id}{\operatorname{id}}

\newcommand{\pd}{\operatorname{pd}}
\newcommand{\Gdim}{\operatorname{Gdim}}

\newcommand{\Ext}{\operatorname{Ext}}
\newcommand{\Supp}{\operatorname{Supp}}

\newcommand{\Tor}{\operatorname{Tor}}
\newcommand{\Hom}{\operatorname{Hom}}

\newcommand{\Zd}{\operatorname{Zd}}
\newcommand{\Ann}{\operatorname{Ann}}
\newcommand{\type}{\operatorname{type}}
\newcommand{\Soc}{\operatorname{Soc}}
\newcommand{\depth}{\operatorname{depth}}

\newcommand{\lo}{\longrightarrow}
\newcommand{\fm}{\frak{m}}
\newcommand{\fp}{\frak{p}}
\newcommand{\fq}{\frak{q}}
\newcommand{\fa}{\frak{a}}

\newcommand{\fn}{\frak{n}}

\newcommand{\htop}{\operatorname{ht}}
\newcommand{\length}{\operatorname{\ell}}
\newcommand{\gr}{\operatorname{gr}}
\newcommand{\m}{\mathfrak m}
\newcommand{\n}{\mathfrak n}

\begin{document}

\author[M. Asgharzadeh]{Mohsen Asgharzadeh }

\title[Modules  of finite injective dimension]{Notes on modules of finite injective dimension}

\address{M. Asgharzadeh }
\email{mohsenasgharzadeh@gmail.com}

\subjclass[2020]{13C14; 13H10; 13D22}

\keywords{Cohen-Macualay rings; complete-intersection; conormal and normal modules; 
	injective dimension; integral domains; regular (Gorenstein) rings; Auslander's zero-divisor; hom and tensor product; regular rings and regular sequences; reflexive modules; syzygy modules.}

\begin{abstract}
	Motivated by Bass' conjecture, we study finitely generated modules of finite injective dimension and the additional constraints they impose on the ambient ring. Beyond ensuring the Cohen--Macaulay property, the presence of such modules enforces further conditions on the ring, including reducedness, normality, being an integral domain, and various singularity conditions such as complete intersection, Gorenstein, and beyond. This continues to detect non-singularity as well.
We also address the reflexivity (and also torsionlessness) of modules with finite injective dimension and show that this forces the ring to be quasi-normal. In the same vein, we investigate the injective dimension of tensor products and endomorphism rings. Finally, we study the behavior of $R$ when high syzygies of $k$ surject onto a non-zero $R$-module of finite injective dimension.
\end{abstract}

\maketitle
\tableofcontents

\section{Introduction}

Bass conjectured in \cite{bass} that the existence of a finitely generated module of finite injective dimension forces a ring to be Cohen--Macaulay. The conjecture was verified in many cases by Peskine and Szpiro \cite{PS} and was ultimately settled by Roberts \cite{int}. Nevertheless, modules of finite injective dimension continue to play a central role in commutative algebra, particularly in connection with the structure of singularities, homological conjectures, and Gorenstein phenomena.

The work of Peskine and Szpiro contains several striking results showing that strong structural properties of a ring may be deduced from the existence of ideals or modules of finite homological dimension.  These observations motivate the guiding theme of this paper:

\begin{center}
	\emph{What structural properties of a ring are forced by the existence of ideals or modules of finite injective dimension?}
\end{center}

The purpose of this paper is to investigate this question from several different perspectives. Our results connect finite injective dimension to reducedness, normality, regularity, reflexivity, Gorenstein conditions, canonical modules, syzygies, and Auslander's zero-divisor philosophy.

\medskip
Recall that, if $\fp\in\Spec(R)$ satisfies
\(
\pd_R(\fp)<\infty,
\)
equivalently $\pd_R(R/\fp)<\infty$, then $R$ is an integral domain by \cite[Corollary II.3.3]{PS}. Moreover, it is easy to see that $R$ is a domain whenever $\id_R(R/\fp)<\infty$. This leads to {Question} \ref{d21}.
The first part of the paper studies prime ideals and low-height ideals of finite injective dimension. In Section~2, we prove several results showing that finite injective dimension imposes strong geometric restrictions on the ring. Among other things, we show:

\begin{theorem}
	\begin{enumerate}
		\item Suppose that $\id_R(\fp)<\infty$ for all $\fp\in\Assh(R)$. Then $R$ is reduced.
		
		\item Suppose $\min(R)$ consists of a single prime. If $\id_R(\fp)<\infty$ for some $\fp\in\Spec(R)$, then $R$ is a domain.
		
		\item Suppose that $\id_R(\fp)<\infty$ for all $\fp\in\Spec(R)$ of height at most one. Then $R$ is normal.
		
		\item Let $R$ be an excellent UFD containing a field $k$ and of dimension greater than three. If $\id_R(\fp)<\infty$ for all $\fp\in\Spec(R)$ of height two, then $R$ is regular.
	\end{enumerate}
\end{theorem}

We also obtain quasi-versions of the reducedness and normality results. For example, we prove the following.

\begin{corollary}
	Suppose that $\id_R(\fp)<\infty$ for some $\fp\in\Spec(R)$ of positive height. Then $R$ is quasi-normal.
\end{corollary}

Stronger conclusions are obtained when finite injective dimension is assumed for larger classes of ideals. For instance, if every ideal of big height at most two has finite injective dimension, then combining the preceding results with a classical theorem of Auslander \cite{baus} yields the regularity of $R$.

Section~3 studies modules of finite injective dimension with few generators. Motivated by Vasconcelos' question in \cite[Page~51]{v1} concerning ``how Gorenstein rings arise'', we investigate reflexive modules of small rank and finite injective dimension. Our first result in this direction is the following.

\begin{proposition}\label{1.1}
	Let $(R,\fm)$ be quasi-normal, and let $M$ be an $R$-module satisfying:
	\begin{enumerate}
		\item[(i)] $\id_R(M)<\infty$,
		\item[(ii)] $M$ is reflexive,
		\item[(iii)] $\mu(M)\leq 2$.
	\end{enumerate}
	Then $R$ is Gorenstein.
\end{proposition}

As an application, we extend Proposition~\ref{1.1} to four-generated modules over rings of type at most two. We also provide additional remarks related to the results of Section~2 and generalize part of Vasconcelos' work \cite[2.40]{v1} to a broader setting.

\medskip
Section~4 develops an injective analogue of Auslander's zero-divisor conjecture. Recall that if $M$ is a nonzero Tor-rigid module, then Auslander proved in \cite[4.3]{au} that every $M$-regular sequence is an $R$-regular sequence. Since modules of finite projective dimension are not necessarily Tor-rigid, Auslander conjectured that the same conclusion nevertheless holds for such modules; this was later proved by Roberts \cite[6.2.3]{int1}. Using two different methods, we establish an injective counterpart of this result and recover Auslander's theorem in the Cohen--Macaulay case. These methods also provide alternative proofs of several results from Section~2.

\medskip

In Section 5, we analyze the conormal module \(I/I^2\) of an ideal \(I\) in a regular ring \(S\) with \(R = S/I\). Our guiding question, inspired by \cite{nano}, is whether finiteness of the injective dimension \(\operatorname{id}_R(I/I^2)\) forces \(I\) to be generated by a regular sequence. We provide several positive answers, mainly in non-Gorenstein cases. 
Since the product of ideals is almost never Gorenstein, we start with:
\begin{theorem}(See 	Theorem \ref{thm:5power}). Suppose $(S,\mathfrak n)$ is a regular local ring, $P$ is a nonzero prime ideal of height $h$, $q\ge 2$, $I=P^q$ and $R=S/I$.
	If $\operatorname{id}_R(I/I^2)<\infty$, then $I$ is principal.
\end{theorem}

Also, we prove that if \(I\) is prime and \(\mu(I) \leq \operatorname{ht}(I)+1\), then \(\operatorname{id}_R(I/I^2) < \infty\) forces a regular sequence. As a consequence, a prime ideal generated by three elements with \(\operatorname{id}_R(I/I^2) < \infty\) is necessarily a complete intersection. This recovers the main theorem of Kunz \cite{kunz}.  
Now, suppose \(R\) is at most \(1\)-dimensional and complete, and let \(N_I := (I/I^2)^*\) be the normal bundle. Then we show
\[
\operatorname{id}_R(N_I) < \infty \;\Longrightarrow\; R \text{ is a complete intersection}.
\]
When \(R\) is zero-dimensional, we show that \(\operatorname{id}_R(N_I) < \infty\) implies \(I\) is generated by a regular sequence without any extra conditions. This shows that the study of the conormal bundle is more subtle compared to the normal bundle. We present a series of nontrivial and interesting examples to support the question. This enables us to compute injective dimension of certain modules. Here, we present only one sample: \begin{example}
Let $k$ be any field and let $n\geq2$. Denote by $I_{n,3}$ the defining ideal of the cubic Veronese ring $k[x_1,\ldots,x_n]^{(3)}$ in its polynomial presentation. Then
	$	\id_{(S/I_{n,3})_{\mathfrak m}}
	\bigl((I_{n,3}/I_{n,3}^{2})_{\mathfrak m}\bigr)=\infty.
	$
\end{example}These use  various computational tricks.
The role of canonical modules as a source of modules of finite injective dimension, emphasized in \cite[2.38]{v1}, leads naturally to questions studied earlier by Matlis. Let $R$ be a one-dimensional domain with fraction field $Q$. Matlis investigated the following.

\begin{question}\label{16}
	\begin{enumerate}
		\item[(i)] When does there exist a surjection $Q/R\twoheadrightarrow E_R(k)$?
		
		\item[(ii)] When is $Q/I\cong E_R(k)$?
	\end{enumerate}
\end{question}

Since part~(ii) is equivalent to the existence of ideals of finite injective dimension, these questions are closely connected to the themes of the present paper. In Section~6, we extend Matlis' results to higher dimensions. We also study quasi-Gorenstein rings, characterized by the condition $H^{\dim(R)}_{\fm}(R)\cong E_R(k)$, and extend an observation of Auslander \cite{comment} concerning the cohomology of $Q/R$, together with a converse statement; see Observation~\ref{CON}.

\medskip
Section~7 is devoted to reflexive and torsionless modules of finite injective dimension. We first show that the existence of a nonzero torsionless module of finite injective dimension forces the completion of $R$ to be generically Gorenstein. This is then strengthened as follows:

\begin{proposition}
	Suppose there exists a nonzero reflexive module $L$ of finite injective dimension. Then $\widehat{R}$ is quasi-normal.
\end{proposition}

As an immediate consequence, the quasi-normality assumption in Proposition~3.1 and its applications can be removed.  It follows easily from the above proposition that
\[
\operatorname{id}_R(N_I) < \infty \;\Longrightarrow\; \widehat{R} \text{ is quasi-normal}.
\]
Using Auslander's transpose and techniques from stable module theory, we prove:

\begin{corollary}
	Let $M$ be reflexive such that $\operatorname{id}_R(M) < \infty$ and $\mu(M) \le 2$. Then $R$ is Gorenstein.
\end{corollary}

We also prove several results relating torsionless modules of finite injective dimension to Serre-type conditions. In particular, if $(G_i)$ denotes the condition that $R_{\fp}$ is Gorenstein for every prime ideal $\fp$ of height at most $i$, then we show:

\begin{theorem}
	Let $R$ be a homomorphic image of a Gorenstein ring. If there exists a nonzero $(k+2)$-torsionless module of finite injective dimension, then $R$ satisfies $(G_k)$.
\end{theorem}

A key ingredient in the proof of the above theorem is a recent result of Kimura \cite{kim} about  Auslander--Reiten conjecture. As a special case, we obtain that a nonzero $(d+2)$-torsionless module of finite injective dimension over a $d$-dimensional Cohen--Macaulay ring forces $R$ to be Gorenstein.
Then we try to sharpening by replacing $(k+2)$-torsionless with $(k+1)$-torsionless, at least in some examples, see e.g., {Observation} \ref{m517}.

Section~8 investigates two questions raised in \cite{q1,q2} that are strongly motivated by the Auslander--Reiten conjecture: if $\Hom_R(M,M)$ has finite injective dimension, must $R$ be Gorenstein? And if both $M$ and $\Ext^i_R(M,M)$ have finite injective dimension for a range of $i$, must $R$ be Gorenstein? Also, see \cite{dey} for  some partial progress toward answering the question of Ghosh--Takahashi and also generalize their main results by reducing the number of vanishing conditions required. Auslander--Reiten conjecture We provide affirmative answers in several cases.
First, we settle both questions in the artinian case using elementary arguments: indecomposable decompositions and the Krull--Remak--Schmidt theorem show that $M$ must be the injective hull $E_R(k)$, whence $R$ is Gorenstein. A modern proof using the vanishing of $\Ext^1_R(M,\Hom_R(M,M))$ and a result of \cite{ta} shows that $M$ is free, again forcing $R$ to be Gorenstein.
We then prove that for a nonzero finite-length module $M$ with $\id_R(M)<\infty$ and $\id_R(\Hom_R(M,M))<\infty$, the ring $R$ is Gorenstein. 
Moving beyond the finite-length case, we prove that if $I$ is a nonzero ideal and $\id_R(\Hom_R(I,I))<\infty$, then $R$ is Gorenstein. The proof proceeds by induction on dimension, using a splitting argument based on depth considerations in codimension one, and ultimately showing that $\Hom_R(I,I)\cong R$. For torsion-free modules $M$ over rings containing $\mathbb{Q}$, we show that $\id_R(\Hom_R(M,M))<\infty$ forces $R$ to be Gorenstein. Recall that taking the integral closure  may realize by $\Hom_R(-,-)$.
As an application, we include a result on the integral closure $\overline{R}$: if $R$ is analytically unramified and $\id_R(\overline{R})<\infty$, then $R$ is normal and Gorenstein. 
Finally, we explore the prime characteristic setting. 
For any  $F$-finite and $F$-injective ring $R$ with quotient singularities, we present the following implication:
	\[
	\id_R(\End_R(\up{\F_n}R))<\infty\Longrightarrow   R \emph{ is regular}.
	\]

Section~9 investigates a question raised in \cite[Question 4.2]{celgo}: if $R$ is a $d$-dimensional local Cohen--Macaulay ring of minimal multiplicity with canonical module $\omega_R$, and $M$ is a maximal Cohen--Macaulay module satisfying $M\otimes_R M^\dagger\cong\omega_R$ (where $M^\dagger=\Hom_R(M,\omega_R)$), must $M$ be isomorphic to $R$ or $\omega_R$? We show that the minimal multiplicity hypothesis maybe be removed. Specifically, 	suppose $R$ is a $d$-dimensional local Cohen--Macaulay ring with canonical module. Let $M$ be a totally reflexive module such that $M \otimes M^\dagger \cong \omega_R$. Then  we show $M \cong R$.

\medskip

Section~10 studies a question of Ghosh--Gupta--Puthenpurakal \cite[Question 1.7]{q3} concerning surjections from syzygy modules onto modules of finite injective dimension. We prove that if
\[
\operatorname{Syz}_j(k)\twoheadrightarrow M
\]
for some nonzero module $M$ of finite injective dimension and some $j\geq \dim(R)$, then $R$ is regular. As an application, we obtain results on isolated Gorenstein singularities for certain Cohen--Macaulay rings admitting surjections onto canonical modules. For motivation and historical background (including the works of Dutta, Takahashi, and Martinkovsky), see the introduction of \cite{q3}.

\medskip
Let us again recall the essential role of canonical module in the land of modules of finite injective dimension.
Finally, in Section~11, we study the question ``When is $\Gdim((\omega_R^\ast)^\dagger)<\infty$?", which was asked in some papers including the work by Holanda and Miranda-Neto.  Here, $M^\dagger$ stands for $\Hom_R(M, \omega_R)$. For example, we show over  \textit{BNSI} rings equipped with a canonical module that
\[
\Gdim((\omega_R^\ast)^\dagger)<\infty \Longrightarrow R \text{ is hypersurface}.
\]
We note that  \textit{BNSI} means every module has Betti numbers strictly increasing, and such rings were introduced by Ramras \cite{ram}.
  Let $(S,\n)$ be a regular local ring,
$0\ne I\subsetneq S$, $R=S/I$, and $C=I/I^2$.
Finally, we investigate:
	\[
	\Gdim_R C<\infty\quad\Longrightarrow\quad \text{ is   R Gorenstein}?
	\]
ally, we investigate:
\(
\Gdim_{S/I}I/I^2<\infty\stackrel{?}\Longrightarrow\quad \text{ R Gorenstein} 
\)
We exclude the zero ideal merely to avoid conventions for the
G-dimension of the zero module. The finite-injective-dimension problem
in the accompanying manuscript asks for the \emph{stronger} conclusion
that $I$ be a complete intersection. The G-dimension condition behaves
differently: over a Gorenstein local ring every finite module has finite
G-dimension, even when its projective dimension is infinite.

\medskip
\section{Primes   of finite injective dimension}

In this paper, \((R,\fm,k)\) is a commutative Noetherian local ring, and modules are finitely generated unless otherwise specified.  
The notation \(\pd_R(-)\) (resp. \(\id_R(-)\)) stands for the projective (resp. injective) dimension of \((-)\).

This section is devoted to studying the following question:

\begin{question}\label{d21}
	If \(R\) admits a prime ideal \(\fp\) such that \(\id_R(\fp) < \infty\), what consequences does this have for the structure of \(R\)?
\end{question}

\begin{notation}
	We denote the set of all associated prime ideals of \((-)\) by \(\Ass_R(-)\).
\end{notation}

\begin{proposition}\label{one}
	Suppose \(\min(R)\) is a singleton. If for some \(\fp\in\Spec(R)\) one has \(\id_R(\fp)<\infty\), then \(R\) is an integral domain.
\end{proposition}

\begin{proof}
	By Bass'  conjecture (see \cite[Section 9.6]{BH}), \(R\) is Cohen--Macaulay. In particular, \(\Ass(R)=\min(R)\).  
	Suppose \(\min(R)=\{P\}\). By definition, there exists an \(x\in R\) such that \((0:_Rx)=P\). Assume \(x\in \fm\).  
	Recall that  
	\[
	x\in\Zd(R)=\bigcup_{\fq\in\Ass(R)}\fq=\bigcup_{\fq\in\min(R)}\fq=P=(0:x).
	\]  
	In other words, \(x^2=0\).  
	Since \(\min(R)\) is a singleton, \(P\subseteq \fp\).  
	Because \(\id_R(\fp)<\infty\), from \cite[Proposition 3.1.9]{BH} we know that \(\id_{R_{\fp}}(\fp R_{\fp})<\infty\).  
	By the Auslander--Buchsbaum--Serre theorem, \(R_{\fp}\) is regular (see \cite[Theorem 2.2.7]{BH}).  
	Since \(R_{\fp}\) is regular, its localizations are also regular. In particular, \(R_{P}=(R_{\fp})_{P R_{\fp}}\) is regular.  
	Recall that \(x\in P\). Since \(x/1\in P R_{P}\), it must be zero. Hence there exists \(y\in R\setminus P\) such that \(xy=0\).  
	From \(xy=0\) we have \(y\in(0:x)=P\), a contradiction. Thus \(x\notin\fm\).  
	Because the ring is local, \(x\) is a unit. Consequently, \(P=(0:x)=0\). Therefore \(R\) is an integral domain.
\end{proof}

We denote by $\Assh_R(-)$ the set of all associated prime ideals $\fp$ such that $\dim R/\fp = \dim R$.

\begin{proposition}
	Suppose that $\id_R(\fp)<\infty$ for all $\fp\in\Assh(R)$. Then $R$ is reduced.
\end{proposition}

\begin{proof}
	Recall that $R$ is Cohen--Macaulay. In particular, $R$ is equidimensional. Let
	\[
	\Assh(R)=\Ass(R)=\min(R)=\{\fp_1,\ldots,\fp_t\}.
	\]
	In view of Proposition \ref{one}, we may assume $t>1$. Let $P\in\min\bigl(\bigcap_{i=1}^t \fp_i\bigr)$. There exists an $i$ such that $\fp_i\subseteq P$. It follows that $P$ is minimal over $\fp_i$. Thus $P=\fp_i$. Since $\id_R(\fp_i)<\infty$, we have $\id_{R_{\fp_i}}(\fp_i R_{\fp_i})<\infty$. By the Auslander--Buchsbaum--Serre theorem, $R_{\fp_i}$ is regular. Then
	\[
	\Bigl(\bigcap_{i=1}^t \fp_i\Bigr)R_P \subseteq P R_P = 0.
	\]
	Since $\min(-)$ denotes the minimal part of $\Supp(-)$, we have shown that $\Supp\bigl(\bigcap_{i=1}^t \fp_i\bigr)=\emptyset$, and therefore $\bigcap_{i=1}^t \fp_i = 0$. Hence $R$ is reduced.
\end{proof}

\begin{notation}
	By $(G_i)$ (resp. $(R_i)$) we mean that $R_\fp$ is Gorenstein (resp. regular) for all $\fp\in\Spec(R)$ of height at most $i$.
\end{notation}
Recall that a module $M$ satisfies $(S_i)$ if $\depth(M_{\fp})\geq \min\{i,\dim(M_\fp)\}$ for all $\fp\in\Spec(R)$.

\begin{proposition}\label{qr}
	Suppose that $\id_R(\fp)<\infty$ for some $\fp\in\Assh(R)$. Then $R$ is quasi-reduced.
\end{proposition}

\begin{proof}
	Following Bass'  conjecture, $R$ is Cohen--Macaulay. In particular, $R$ satisfies Serre's condition $(S_1)$.
	Let $\fq\in\Ass(R)=\min(R)$. We have two possibilities: (i) $\fq=\fp$ or (ii) $\fq\neq\fp$.
	
	\begin{enumerate}
		\item[(i)] In the first case, $R_{\fq}$ is a field. To see this, recall that $\id_{R_{\fq}}(\fq R_{\fq})=\id_{R_{\fp}}(\fp R_{\fp})$ is finite, because $\id_R(\fp)<\infty$. This in turn implies that $R_{\fq}$ is zero-dimensional, local, and regular. Hence $R_{\fq}$ is a field.
		
		\item[(ii)] In the second case, $R_{\fq}$ is Gorenstein. Indeed, we have $\fq R_{\fq} = \fp R_{\fq}$. Consequently, $\id_{R_{\fq}}(R_{\fq}) = \id_{R_{\fq}}(\fq R_{\fq})$, and the latter is finite because $\fp$ has finite injective dimension. Thus $R_{\fq}$ is Gorenstein.
	\end{enumerate}
	
	In both cases, $R_{\fq}$ is Gorenstein. It remains to note that $(S_1)$ together with $(G_0)$ implies the quasi-reduced condition.
\end{proof}

\begin{proposition}\label{nor}
	Suppose that $\id_R(\fp)<\infty$ for all $\fp\in\Spec(R)$ of height at most one. Then $R$ is normal.
\end{proposition}

\begin{proof}
	Recall from the assumptions that $R$ satisfies Serre's condition $(S_2)$. Let $\fp\in\Spec(R)$ be of height at most one. Since $\id_R(\fp)<\infty$, we have $\id_{R_{\fp}}(\fp R_{\fp})<\infty$. By the Auslander--Buchsbaum--Serre theorem, $R_{\fp}$ is regular. Hence $R$ satisfies $(R_1)$. It remains to note that a ring satisfying Serre's conditions $(S_2)$ and $(R_1)$ is normal.
\end{proof}

We cite \cite{wol} as our reference for quasi-normal rings.

\begin{proposition}\label{qn}
	Suppose that $\id_R(\fp)<\infty$ for some $\fp\in\Spec(R)$ of positive height. Then $R$ is quasi-normal.
\end{proposition}

\begin{proof}
	As in Proposition \ref{nor}, the ring $R$ satisfies $(S_2)$.
	Let $\fq\in\Spec(R)$ be of height at most one. We may assume that $\fp$ is not maximal; otherwise the ring is regular. We have two possibilities:
	(i) $\fq\subseteq \fp$, or (ii) $\fq\nsubseteq \fp$.
	
	\begin{enumerate}
		\item[(i)] We localize at $\fp$. Using the assumption $\id_R(\fp)<\infty$, we know that $R_{\fp}$ is regular (by the Auslander--Buchsbaum--Serre theorem). Since $\fq\subseteq \fp$, we have $R_{\fq}\cong (R_{\fp})_{\fq R_{\fp}}$. As a localization of a regular ring is regular, we deduce that $R_{\fq}$ is regular. Hence $R$ satisfies $(R_1)$.
		
		\item[(ii)] Suppose $\fq\nsubseteq\fp$. We claim that $\fp\nsubseteq\fq$. Indeed, if $\fp\subseteq \fq$, then, because $\operatorname{ht}(\fp)\geq 1$ and $\operatorname{ht}(\fq)\leq 1$, we must have $\fp=\fq$. This contradicts the assumption $\fq\nsubseteq\fp$. Therefore $\fp\nsubseteq\fq$, which gives the natural isomorphism $\fp R_{\fq}\cong R_{\fq}$. Since $\id_R(\fp)<\infty$, its localizations also have finite injective dimension. In particular, $\id_{R_{\fq}}(R_{\fq})=\id_{R_{\fq}}(\fp R_{\fq})<\infty$. By definition, $R_{\fq}$ is Gorenstein. Thus $R$ satisfies $(G_1)$.
	\end{enumerate}
	
	In both cases, $R$ is $(G_1)$. It remains to note that a ring satisfying $(S_2)$ and $(G_1)$ is quasi-normal.
\end{proof}

\begin{corollary}\label{pgen}
	Suppose that $\id_R(\fp)<\infty$ for some $\fp\in\Spec(R)$. Then $R$ is quasi-reduced.
\end{corollary}

\begin{proof}
	If $\fp\in\Assh(R)$, then the desired claim is the subject of Proposition \ref{qr}. Thus, without loss of generality, we may assume that $\dim(R)>0$ and that $\fp\in\Spec(R)\setminus\Assh(R)$. By Bass'  conjecture, $R$ is Cohen--Macaulay. Since the ring is equidimensional, we have $\operatorname{ht}(\fp)>0$. It remains to apply Proposition \ref{qn}.
\end{proof}

\begin{proposition}
	Let $(R,\fm,k)$ be an excellent UFD containing $k$ and of dimension $>3$. Suppose that $\id_R(\fp)<\infty$ for all $\fp\in\Spec(R)$ of height two. Then $R$ is regular.
\end{proposition}

\begin{proof}
	It is straightforward to see that $R$ satisfies Serre's conditions $(S_3)$ and $(R_2)$. Let $M:=\operatorname{Syz}_2(k)$. Then $M$ is reflexive and locally free over the punctured spectrum. According to Miller's paper \cite{MILL}, there exist a free module $F$ and a prime ideal $\fp$ of height at most two that fit into the following Bourbaki sequence:
	\[
	0 \longrightarrow F \longrightarrow M \longrightarrow \fp \longrightarrow 0 \qquad (\ast)
	\]
	Recall that $R$ is Cohen--Macaulay. Thanks to a result of Murthy \cite[3.3.19]{BH}, $R$ is Gorenstein. This shows that $\id_R(F)<\infty$. Applying this to $(\ast)$, we see that $\id_R(M)$ is finite. Hence $\id_R(k)<\infty$. By the Auslander--Buchsbaum--Serre theorem, $R$ is regular.
\end{proof}

What about rings of mixed characteristic? More formally, we ask:

\begin{problem}
	How can one find the mixed characteristic analogue of the Bourbaki sequence and of the related prime ideals?\footnote{We are interested in this, even if the ring is regular.}
\end{problem}

If we allow more ideals to have finite injective dimension, then we present the following three positive answers:

\begin{proposition}\label{baus}
	Let $(R,\fm)$ be local. Suppose that $\id_R(\fa)<\infty$ for all ideals of big-height at most two. Then $R$ is regular.
\end{proposition}

\begin{proof}
	By Proposition \ref{nor}, $R$ is normal. We assume that $\dim R>0$. Thanks to Bass, $R$ is Cohen--Macaulay. 
	For any nonzero $x\in\fm$, $(xR)$ is height-unmixed and of height one. By assumption, $\id_R(xR)$ is finite. Since $R$ is a domain, $xR\cong R$. Thus $\id(R)<\infty$. In other words, $R$ is Gorenstein. This allows us to assume that $\pd_R(\fa)<\infty$ for every ideal $\fa$ of big-height at most two. This, together with normality and the Cohen--Macaulay property, enables us to apply a result of Auslander \cite[Corollary 5]{baus} and deduce that $R$ is regular.
\end{proof}

\begin{proposition}
	Let $(R,\fm,k)$ be a complete UFD. Suppose that $\id_R(\fa)<\infty$ for every unmixed ideal of height two. Then $R$ is regular.
\end{proposition}

\begin{proof}
	We may assume that $\dim(R)>1$. Let $x,y$ be a parameter sequence. Then $(x,y)$ is height-unmixed. By assumption, $\id_R(x,y)<\infty$. Again, Bass'  conjecture implies that the ring is Cohen--Macaulay. By the aforementioned result of Murthy, $R$ is Gorenstein. Hence $\pd_R(\fa)<\infty$ for every unmixed ideal of height two. It remains to apply the discussion below \cite[Corollary 5]{baus}.
\end{proof}
Here is another variation of Proposition \ref{baus}:

\begin{observation}
	Suppose any two-generated ideal of $R$ has finite injective dimension. Then $R$ is a Gorenstein UFD.
\end{observation}

\begin{proof}
	The ring $R$ itself is a two-generated ideal (since $R = (1)$) and therefore has finite injective dimension. Hence $R$ is Gorenstein.\footnote{Alternatively, see \cite[II.5.5]{PS}.}
	Over Gorenstein rings, finite injective dimension implies finite projective dimension (see \cite[3.1.25]{BH}). The result then follows by \cite[5.3]{BE}.
\end{proof}

\section{Modules with few generators and finite injective dimension}

We start with the following easy exercise:

\begin{fact}
	Let $R$ be a local domain. Suppose there is a cyclically presented module $M$ of finite injective dimension. Then $R$ is Gorenstein.
\end{fact}

\begin{proof}
	Let $M = R/xR$ for some $x$. Set $i := \id_R(M) < \infty$ and consider the exact sequence $0 \to R \xrightarrow{x} R \to M \to 0$. This yields
	\[
	\Ext_R^{i+1}(k,R) \xrightarrow{x} \Ext_R^{i+1}(k,R) \longrightarrow \Ext_R^{i+1}(k,M) = 0.
	\]
	By Nakayama's lemma, $\Ext_R^{i+1}(k,R) = 0$, and hence $\id(R) < \infty$. By definition, $R$ is Gorenstein.
\end{proof}

The following is nontrivial:

\begin{fact}(Peskine--Szpiro).\label{cy}
	Suppose there is a cyclic module $M$ of finite injective dimension. Then $R$ is Gorenstein.
\end{fact}

\begin{corollary}\label{gd}
	Suppose $R$ is Gorenstein. If for some $\fp\in\Spec(R)$ one has $\id_R(\fp)<\infty$, then $R$ is an integral domain.
\end{corollary}

\begin{proof}
	We have $\pd_R(\fp)<\infty$ (since over a Gorenstein ring, finite injective dimension implies finite projective dimension). Then $R$ is an integral domain by Proposition \ref{one} (or by the argument therein).
\end{proof}

\begin{corollary}
	Suppose that for some principal prime ideal $\fp\in\Spec(R)$ we have $\id_R(\fp)<\infty$. Then $R$ is an integral domain.
\end{corollary}

\begin{proof}
	By Fact \ref{cy}, $R$ is Gorenstein. Then Corollary \ref{gd} implies that $R$ is an integral domain.
\end{proof}

Here we present a new proof of \cite[Theorem 5.2]{v1} and \cite[Theorem 2.11(iv)]{sha}.

\begin{corollary}
	Let $I \lhd R$ be nonzero. If $\id_R(R/I) < \infty$, then $I$ contains an $R$-regular element.
\end{corollary}

\begin{proof}
	By Fact \ref{cy}, $R$ is Gorenstein. By \cite[3.1.25]{BH}, $\pd_R(R/I) < \infty$. Equivalently, $I$ has a finite free resolution. By a famous result of Burch \cite[Corollary 1.4.7]{BH}, $I$ contains a regular element.
\end{proof}

\begin{notation}
	By $\mu(-)$ we denote the minimal number of generators of a finitely generated module $(-)$.
\end{notation}
		\begin{fact}\label{fst} (See  \cite[Section 5.3]{finitsup}). Let  $(R,\fm)$ be a  generically Gorenstein  Cohen-Macaulay local ring  possessing a canonical module.  Suppose $R$  is  of type at most two. Then $\Ext^1_R(\omega_R,R) =0$
	if and only if $R$ is Gorenstein. 
\end{fact}
In Section 7, we extend the next result by removing the quasi-normality hypothesis.

\begin{proposition}
	Let $(R,\fm)$ be quasi-normal, and let $M$ be such that
	\begin{enumerate}
		\item[(i)] $\id_R(M) < \infty$,
		\item[(ii)] $M$ is reflexive,
		\item[(iii)] $\mu(M) \leq 2$.
	\end{enumerate}
	Then $R$ is Gorenstein.
\end{proposition}

\begin{proof}
	By Fact \ref{cy} we may assume that $M$ is not cyclic. In light of Bass'  conjecture, $R$ is Cohen--Macaulay. Let $d := \dim(R)$, and let $\underline{x} := x_1,\ldots,x_d$ be a system of parameters. In light of \cite[II.5.4]{PS}, we have
	\[
	\Ext_R^d\bigl(\Hom_R(k,R/\underline{x}R), M\bigr) \cong k \otimes_R \frac{M}{\underline{x}M} \cong \frac{M}{\fm M}.
	\]
	Let $\type(R)$ denote the type of $R$. From this we deduce that $\displaystyle \bigoplus_{\type(R)} \Ext_R^d(k,M) \cong k \oplus k$. Let $\mu_i(M)$ denote the $i$th Bass number of $M$, i.e., the number of copies of $E_R(k)$ in a minimal injective resolution of $M$. In this terminology, we obtain $\mu_d(M) \cdot \type(R) = 2$.
	
	If $\type(R) = 1$, there is nothing to prove, as rings of type one are Gorenstein. Thus, without loss of generality, $\type(R) = 2$. This implies that $\mu_d(M) = 1$. Since $M$ is torsion-free, it has maximal dimension. Then, in view of \cite[3.4.c]{fo}, we observe that $M$ is maximal Cohen--Macaulay. Without loss of generality, we may assume that the ring is complete. Over a Cohen--Macaulay and complete ring, the canonical module exists and is unique; we denote it by $\omega_R$. Since $M$ is maximal Cohen--Macaulay and of finite injective dimension, we use a result of Sharp \cite{sha} to obtain the identification $M \cong \bigoplus \omega_R$ (see also \cite[3.3.28]{BH}). Since $\mu(M) = 2$ and $\mu(\omega_R) = \type(R) = 2$, we deduce that $M \cong \omega_R$.
	
Since $R$ is generically Gorenstein, we may assume $\omega_R$ is an ideal of $R$.	Let $(-)^\ast$ denote $\Hom_R(-,R)$. Thanks to \cite[Appendix]{hh}, there is an exact sequence
	\[
	0 \longrightarrow \omega^\ast \longrightarrow R^2 \longrightarrow \omega_R \longrightarrow 0 \qquad (\ast)
	\]
	Recall that $\Ext_R^1(R^2,R) = 0$. Applying $\Hom_R(-,R)$ to the short exact sequence $(\ast)$ yields the following commutative diagram:
	\[
	\begin{CD}
	0 @>>> \omega^\ast @>>> R^2 @>>> \omega^{\ast\ast} @>>> \Ext_R^1(\omega_R,R) @>>> 0 \\
	@. @V{\cong}VV @V{\cong}VV @V{\cong}VV @. \\
	0 @>>> \omega^\ast @>>> R^2 @>>> \omega @>>> 0\\
\end{CD}
	\]
	From this, we conclude that $\Ext_R^1(\omega_R,R) = 0$. Now, we apply {Fact} \ref{fst}  to deduce  $R$ is Gorenstein, as desired claimed.
\end{proof}

The above proof can be deduced from the following nice observation:

\begin{observation}
	Suppose $N$ is finitely generated and of finite injective dimension. Then $\mu(N) \geq \operatorname{type}(R)$.
\end{observation}

\begin{proof}
	The ring $R$ is Cohen--Macaulay, and we may assume it is complete. Consider the maximal Cohen--Macaulay module $M := R$. By \cite[Corollary 3.15]{celp}, there is a surjection $M^{\oplus \mu(N)} \to \omega_R \to 0$. It remains to note from \cite[3.3.11(c)]{BH} that $\mu(\omega_R) = \type(R)$.
\end{proof}
\begin{proposition}\label{pqn}
	Let $(R,\fm)$ be quasi-normal of type at most two, and let $M$ be an indecomposable module such that
	\begin{enumerate}
		\item[(i)] $\id_R(M) < \infty$,
		\item[(ii)] $M$ is reflexive,
		\item[(iii)] $\mu(M) \leq 4$.
	\end{enumerate}
	Then $R$ is Gorenstein.
\end{proposition}

\begin{proof}
	We may assume that $\mu(M) > 3$. Recall from (i) that $R$ is Cohen--Macaulay. Suppose, for contradiction, that $R$ is not Gorenstein. Then we may assume that $\type( R) = 2$. By the above argument,
	\[
	2\mu_d(M) = \mu_d(M) \cdot \type( R) = \mu(M).
	\]
	Hence $\mu(M)$ is even. In fact, $\mu(M) = 4$ and so $\mu_d(M) = 2$. Recall that $M$ satisfies $(S_2)$. According to \cite[Theorem 3]{Ao}, $M$ is Cohen--Macaulay, and hence maximal Cohen--Macaulay. Since $M$ has finite injective dimension, $M \cong \bigoplus \omega_R$. Because $M$ is indecomposable, $M \cong \omega_R$. But then $4 = \mu(M) = \mu(\omega_R) = \type(R) = 2$, a contradiction. This completes the proof.
\end{proof}

\begin{corollary}
	Let $(R,\fm)$ be quasi-normal of type two, and let $M$ be a module such that $\id_R(M) < \infty$. Then $\mu(M)$ is even.
\end{corollary}

\begin{remark}
	In the previous three results, we may replace the quasi-normal assumption with the integral domain assumption; see \cite[Corollary 5.11]{finitsup}.
\end{remark}
The following is dual to \cite[6.2.4]{int1} and is clear from Bass'  conjecture:

\begin{observation}
	Suppose there is a finite length module $A$ of finite injective dimension. Then $R$ is Cohen--Macaulay.
\end{observation}

\begin{proof}
	We may assume $R$ is complete. By Matlis duality, $M := \Hom_R(A, E_R(k))$ has finite length and finite flat dimension. By \cite{GR}, $\pd_R(M) < \infty$. Since $\ell(M \otimes R) < \infty$ and by the intersection theorem, $\dim R \leq \pd_R(M) = \depth(R) - \depth(M) = \depth(R)$, as claimed.
\end{proof}

By the same reasoning, the following slightly extends \cite[6.2.4]{int1}:

\begin{observation}\label{art}
	Suppose there is an artinian module $A$ of finite projective dimension. Then $R$ is Cohen--Macaulay.
\end{observation}

By applying more advanced results, we show:

\begin{observation}\label{art2}
	Suppose there is an $\fm$-torsion module $A$ of finite projective dimension. Then $R$ is Cohen--Macaulay.
\end{observation}

\begin{proof}
	Since $A$ is $\fm$-torsion, $A^\vee$ is complete (see \cite[4.2]{sim}). In particular, there exists an acyclic complex of injective modules resolving the complete module $A^\vee$. Recall from the work of André that $R$ is equipped with a big Cohen--Macaulay module. By \cite[7.6]{sim}, $R$ is Cohen--Macaulay.
\end{proof}

The following drops the Cohen--Macaulay assumption from \cite[Proposition 2.40]{v1}:

\begin{corollary}\label{vps}
	Let $(R,\fm)$ be a $2$-dimensional ring, and let $I$ be an ideal of grade two and finite projective dimension. Then $\type(R/I) = \type(R)(\mu(I) - 1)$.
\end{corollary}

\begin{proof}
	Recall from $2 = \dim R \geq \operatorname{ht}(I) \geq \operatorname{grade}(I,R) = 2$ that $I$ is $\fm$-primary. From $0 \to I \to R \to R/I \to 0$ we deduce that the artinian module $R/I$ has finite projective dimension. By Observation \ref{art}, $R$ is Cohen--Macaulay. The exact sequence
	\[
	\cdots \to H^i_{\fm}(I) \to H^i_{\fm}(R) \to H^i_{\fm}(R/I) \to \cdots
	\]
	gives $\depth(I) = 1$. By the Cohen--Macaulay property and the Auslander--Buchsbaum formula, $\pd_R(I) = \depth(R) - \depth(I) = 1$. These facts allow us to apply \cite[Proposition 2.40]{v1} and conclude that $\type(R/I) = \type(R)(\mu(I) - 1)$.
\end{proof}

Peskine--Szpiro \cite[Corollary II.5.7]{PS} study the case $\mu(I) = 2$. Here is the case $\mu(I) = 3$:

\begin{corollary}
	Let $(R,\fm)$ be a $2$-dimensional ring, and let $I$ be an ideal of grade two and finite injective dimension. If $\mu(I) = 3$, then:
	\begin{enumerate}
		\item[(i)] $R$ is Gorenstein,
		\item[(ii)] $R/I$ is Cohen--Macaulay and of type two.
	\end{enumerate}
\end{corollary}

\begin{proof}
	(i) Recall that $R$ is Cohen--Macaulay. Let $\underline{x} = x_1, x_2$ be a parameter sequence. Then
	\[
	\Ext^2_R\bigl(\Hom_R(k, R/\underline{x}R), I\bigr) \cong \frac{I}{\fm I}.
	\]
	This gives $\type(R) \cdot \mu_2(I) = 3$. Recall that $\depth(I) = 1 < 2 = \dim(I)$. Thanks to \cite[Proposition 3.8]{fo}, we know that $\mu_2(I) \geq 2$. Substituting this into the previous equality yields $\type(R) = 1$; i.e., $R$ is Gorenstein.
	
	(ii) Since $R/I$ is zero-dimensional, it is Cohen--Macaulay. From part (i) we know that $R$ is Gorenstein. This implies that $\pd_R(I) < \infty$. By Corollary \ref{vps}, $R/I$ is of type two.
\end{proof}

\section{Auslander's zero-divisor conjecture}

As a basic tool, we prepare a version of Auslander's zero-divisor conjecture for modules of finite injective dimension. This has its own importance.

\begin{proposition}\label{AUS}
	Let $M$ be finitely generated and of finite injective dimension. Then any $M$-regular element is $R$-regular.
\end{proposition}

\begin{proof}
	By \cite[Proposition 3.6]{cel}, it suffices to show that $\operatorname{grade}(\fp, M) \leq \operatorname{grade}(\fp, R)$ for all $\fp \in \Spec(R)$. 
	Following Bass'  conjecture, $R$ is Cohen--Macaulay. Then
	\[
	\begin{array}{rcl}
	\operatorname{grade}(\fp, M) & \stackrel{(1)}{\leq} & \depth(M_{\fp}) \\
	& \stackrel{(2)}{\leq} & \dim(M_{\fp}) \\
	& \stackrel{(3)}{\leq} & \id_{R_{\fp}}(M_{\fp}) \\
	& \stackrel{(4)}{=} & \depth(R_{\fp}) \\
	& \stackrel{(5)}{=} & \operatorname{grade}(\fp, R),
	\end{array}
	\]
	where (1) is from \cite[Exercise 16.5]{mat}; (2) is from \cite[Proposition 1.2.12]{BH}; (3) and (4) are from \cite[Theorem 3.1.17]{BH}; and (5) follows from \cite[Theorem 2.1.3(b)]{BH} together with the fact that $R$ is Cohen--Macaulay.
\end{proof}

Let us reconstruct Proposition \ref{AUS} by a new argument:

\begin{proposition}\label{sh}
	Let $M$ be finitely generated and of finite injective dimension. Then any $M$-regular element is $R$-regular.
\end{proposition}

\begin{proof}
	Let $x \in \fm$ be $M$-regular. Then $x$ is regular over $M \otimes_R \widehat{R}$ because completion is flat. It is easy to see that $\id_{\widehat{R}}(M \otimes_R \widehat{R}) < \infty$. In summary, $\widehat{R}$ is Cohen--Macaulay and admits a canonical module.
	Recall that
	\[
	\begin{array}{rcl}
	\Ass_{\widehat{R}}\bigl(\Hom_{\widehat{R}}(\omega_{\widehat{R}}, M \otimes_R \widehat{R})\bigr)
	& \stackrel{(1)}{=} & \Supp(\omega_{\widehat{R}}) \cap \Ass_{\widehat{R}}(M \otimes_R \widehat{R}) \\
	& \stackrel{(2)}{=} & \Spec(\widehat{R}) \cap \Ass_{\widehat{R}}(M \otimes_R \widehat{R}) \\
	& \stackrel{(3)}{=} & \Ass_{\widehat{R}}(M \otimes_R \widehat{R}),
	\end{array}
	\]
	where
	\begin{enumerate}
		\item[(1)] see \cite[1.2.27]{BH};
		\item[(2)] see \cite[Theorem 3.3.5(b)]{BH};
		\item[(3)] is trivial.
	\end{enumerate}
	Also,
	\[
	\begin{array}{rcl}
	\Zd\bigl(\Hom_{\widehat{R}}(\omega_{\widehat{R}}, M \otimes_R \widehat{R})\bigr)
	& = & \bigcup_{\fq \in \Ass(\Hom_{\widehat{R}}(\omega_{\widehat{R}}, M \otimes_R \widehat{R}))} \fq \\
	& = & \bigcup_{\fq \in \Ass(M \otimes_R \widehat{R})} \fq \\
	& = & \Zd(M \otimes_R \widehat{R}),
	\end{array}
	\]
	and since $x \notin \Zd(M \otimes_R \widehat{R})$, we deduce that $x$ is regular over $\Hom_{\widehat{R}}(\omega_{\widehat{R}}, M \otimes_R \widehat{R})$.
	It is easy to see from \cite{sha} that $\pd_{\widehat{R}}\bigl(\Hom_{\widehat{R}}(\omega_{\widehat{R}}, M \otimes_R \widehat{R})\bigr)$ is finite; see also \cite[9.6.5(b)]{BH}. By Auslander's zero-divisor theorem, $x$ is $\widehat{R}$-regular (see \cite[Theorem 9.4.7]{BH}). Since $R \subseteq \widehat{R}$, $x$ is $R$-regular. This is what we wanted to prove.
\end{proof}

Let us present a modern proof of Proposition \ref{one}:

\begin{corollary}\label{csing}
	Suppose $\min(R)$ is a singleton. If for some $\fp \in \Spec(R)$ we have $\id_R(\fp) < \infty$, then $R$ is an integral domain.
\end{corollary}

\begin{proof}
	We know $R$ is Cohen--Macaulay, and in particular $\min(R) = \Ass(R)$. Let $\min(R) = \{P\}$.
	Let $\pi : R \to R_{\fp}$ be the natural localization map sending $r$ to $r/1$. Let $r \in \ker(\pi)$. By definition, there exists $x \in R \setminus \fp$ such that $rx = 0$. We claim that $x$ is $\fp$-regular. Indeed, if not, then
	\[
	x \in \Zd(\fp) = \bigcup_{\fq \in \Ass(\fp)} \fq \subseteq \bigcup_{\fq \in \Ass(R)} \fq = P.
	\]
	Since $\min(R) = \{P\}$, we have $P \subseteq \fp$. Hence $x \in \fp$, a contradiction. Thus $x$ is an $\fp$-regular element.
	Now we apply the assumption together with Proposition \ref{sh} to see that $x$ is $R$-regular. It follows that $r = 0$, i.e., $\pi$ is injective.
	From $\id_R(\fp) < \infty$ we know that $R_{\fp}$ is regular and local. Regular local rings are domains (see \cite[Proposition 2.2.3]{BH}). Hence $R_{\fp}$ is an integral domain. Since $R \subseteq R_{\fp}$, we obtain the desired conclusion.
\end{proof}

Let us reprove Auslander's zero-divisor conjecture over Cohen--Macaulay rings:

\begin{observation}\label{NEW}
	Let $(R,\fm)$ be Cohen--Macaulay, and let $M$ be finitely generated with $\pd_R(M) < \infty$. If $x$ is $M$-regular, then $x$ is $R$-regular.
\end{observation}

\begin{proof}
	We may assume that $R$ is complete. In particular, $\omega_R$ exists.
	The assumption $\pd_R(M) < \infty$ implies that $\id_R(M \otimes_R \omega_R) < \infty$.
	We proved in the previous proposition that $\Ass(M \otimes_R \omega_R) = \Ass(\Hom_R(\omega_R, M \otimes_R \omega_R))$.
	We now use \cite[Theorem 2.9]{sha} to deduce that $\Hom_R(\omega_R, M \otimes_R \omega_R) \cong M$. Combining this with the previous observation, we obtain $\Ass(M \otimes_R \omega_R) = \Ass(M)$. Hence $x$ is $(M \otimes_R \omega_R)$-regular.
	Thanks to Proposition \ref{AUS}, $x$ is $R$-regular, as claimed.
\end{proof}

\begin{remark}
	(i) Recall that \cite{moh3} asks for the computation of associated prime ideals of tensor products. As an example, we proved in the setting of Observation \ref{NEW} that $\Ass(M \otimes_R \omega_R) = \Ass(M)$. In particular, we reproved \cite[Corollary 4.6]{moh3} by a different argument.
	
	(ii) By symmetry, in the proof of Observation \ref{NEW}, one may use \cite[Corollary 4.6]{moh3} instead of \cite[Theorem 2.9]{sha}.
\end{remark}

Recall from \cite{au} that a module $T$ is said to be tor-rigid if there exists a non-negative integer $n$ such that for every finitely generated $R$-module $M$, the vanishing of $\Tor_n^R(T,M)$ implies that $\Tor_{n+i}^R(T,M)$ vanishes for all $i \geq 0$. By Auslander \cite[4.3]{au}, this implies the zero-divisor property.

\begin{remark}\label{NOT}
	Despite the useful duality between modules of finite projective dimension and modules of finite injective dimension, let us present a different situation. It is easy to see that modules of projective dimension one are tor-rigid. However, modules of injective dimension one are not necessarily tor-rigid. This may occur even over $1$-dimensional integral domains of type two. For instance, let $R := k[[x^3, x^4, x^5]]$. Then $(x^3, x^4)$ has injective dimension one but is not tor-rigid.
\end{remark}

\medskip
\section{Conormal modules of finite injective dimension  and related variants}

In this section, $(S,\fn)$ is regular, $I \lhd S$ is an ideal, and $R := S/I$. We study the following:

\begin{question}\label{q51}
	\begin{enumerate}
	\item[(i)] (See \cite{nano}). Suppose $\id_R(I/I^2)$ is finite. Is $I$ generated by a regular sequence?
	\item[(ii)] When is $\id_R(I/I^{(2)}) < \infty$? 
		\item[(iii)] When is $\id_R(N_I) < \infty$? 
\end{enumerate}	
\end{question}

Here, $I^{(n)}$ is the $n$-th symbolic power of $I$. Also, $N_I:=(I/I^2)^\ast$ is called the normal bundle, and it plays a role in algebra and geometry.

\begin{lemma} \label{lem:artin}
	Let $(A,\m,k)$ be an Artinian local ring and let $0\ne M$ be a finite $A$-module. If $\id_A M<\infty$, then
	\[
	M\cong E_A(k)^{\oplus s},\qquad
	s=\dim_k\Soc_A M,
	\qquad
	\mu_A(M)=s\,\type(A)\geq s.
	\]
\end{lemma}
\begin{proof}
	The Bass formula gives $\id_A M=\depth A=0$, hence $M$ is injective. Every finite injective module over an Artinian local ring is a direct sum of copies of the canonical module $E_A(k)$, whose socle has dimension one and whose minimal number of generators is $\type(A)$.
\end{proof}

\begin{lemma}\label{lem:param}
	Let $(T,\m,k)$ be a regular local ring of dimension $h\geq1$. Let $Q\subseteq T$ be an $\m$-primary ideal generated by a regular sequence, and let $q\geq2$. Let
$A=T/Q^q$ and $ C=Q^q/Q^{2q}.$
	Then
	\[
	\mu_A(C)=\binom{h+q-1}{h-1},\qquad
	\dim_k\Soc_A(C)=\binom{h+2q-2}{h-1}.
	\]
	In particular, if $h\geq2$ then $\id_A C=\infty$.
\end{lemma}
\begin{proof}
	Since $Q$ is an $\m$-primary complete-intersection ideal, its generators form a system of parameters and
$\gr_Q T\cong (T/Q)[Y_1,\ldots,Y_h] $.
	Thus $Q^a/Q^{a+1}$ is free over $T/Q$ of rank $\binom{h+a-1}{h-1}$. As $Q^{2q}\subseteq\m Q^q$ and $Q^{q+1}\subseteq\m Q^q$, reduction modulo $\m$ gives
	\[
	\mu_A C=\mu_T(Q^q)=\binom{h+q-1}{h-1}.
	\]
	We next compute the socle with respect to \emph{$\m$}, not merely with respect to $Q$. Formula $\gr_Q T\cong (T/Q)[Y_1,\ldots,Y_h] $ implies
$(Q^{2q}:_T Q)=Q^{2q-1}$.
	Indeed, if $z\notin Q^{2q-1}$, its nonzero initial form of $Q$-adic order at most $2q-2$ stays nonzero on multiplication by a parameter generator, since every $Y_i$ is regular in the polynomial ring $\gr_Q T\cong (T/Q)[Y_1,\ldots,Y_h] $. An element of $C$ annihilated by $\m$ is annihilated by $Q$ and hence, by $(Q^{2q}:_T Q)=Q^{2q-1}$, lies in $Q^{2q-1}/Q^{2q}$. Conversely, on this last graded piece the $\m$-socle is the socle over $T/Q$. Since $T/Q$ is artinian Gorenstein,
	\[
	\Soc_A C=\Soc_T(Q^{2q-1}/Q^{2q}),\qquad
	\dim_k\Soc_A C=\binom{h+2q-2}{h-1}.
	\]
	For $h\ge2$ and $q\ge2$ this socle dimension strictly exceeds $\mu_A C$, contradicting Lemma~\ref{lem:artin} if the injective dimension were finite.
\end{proof}

\begin{theorem} \label{thm:powers}
		Let $(S,\n)$ be a regular local ring and let $0\ne J\subsetneq S$ be an ideal such that $JS_P$ is generated by a regular sequence for every $P\in\operatorname{Min}(S/J)$. For every integer $q\geq2$, 		let $I:=J^q$ and $R=S/I$. Then
$ \id_R(I/I^2)<\infty$ implies 
		J{ is principal}. 
\end{theorem}
\begin{proof}
	Assume $\id_R(I/I^2)<\infty$. The module $I/I^2$ is nonzero, so the Bass conjecture (the theorem of Roberts) implies $R$ is Cohen--Macaulay. Let $P$ be any minimal prime of $J$, and put $h=\operatorname{ht}P$. The localization $T=S_P$ is regular local of dimension $h$; the ideal $Q=JS_P$ is $PS_P$-primary and, by assumption, is generated by a regular sequence. Localization gives
	\[
	R_P=T/Q^q,\qquad (I/I^2)_P=Q^q/Q^{2q}.
	\]
	Finite injective dimension localizes, so Lemma~\ref{lem:param} forces $h=1$. Since $P$ was arbitrary, in particular $\operatorname{ht}I=1$. As $R$ is Cohen--Macaulay,
	\[
	\pd_S R=\depth S-\depth_S R
	=\dim S-\dim R=1
	\]
	by Auslander--Buchsbaum. The exact sequence $0\to I\to S\to R\to0$ makes $I$ a projective $S$-module and hence a rank-one free module. Thus $J^q=I$ is principal.
	
	For completeness, a power of a nonzero ideal $J$ can be principal in the regular local ring $S$ only if $J$ is principal. Choose a greatest common divisor $g$ of a finite generating set of $J$ in the factorial ring $S$, and write $J=gK$ where the generators of $K$ have gcd~$1$. If $J^q$ is principal, so is $K^q$, say $K^q=(f)$. As $a^q\in(f)$ for every generator $a$ of $K$, the element $f$ divides their $q$th powers and hence is a unit. Therefore $K^q=S$, so $K=S$, and $J=(g)$.
	
	Conversely, if $J=(g)\ne0$, then $I=(g^q)$, the conormal module $I/I^2\cong S/I=R$, and $R$ is a Gorenstein hypersurface. Hence $\id_R(I/I^2)=\dim R<\infty$.
\end{proof}
\begin{corollary}\label{cor:radical}
	Theorem~\ref{thm:powers} applies in particular when $J$ is a nonzero radical ideal, or when $J$ itself is generated by a regular sequence. Thus for $q\geq2$, finite injective dimension of the conormal module of $J^q$ forces $J$ to be principal in either case.
\end{corollary}
\begin{proof}
	If $J$ is radical and $P$ is a minimal prime of $J$, then $JS_P=PS_P$, the maximal ideal of the regular local ring $S_P$. This is a parameter ideal. For a complete-intersection $J$, every localization is a complete intersection.
\end{proof}

We need the following easy fact:

\begin{fact}\label{tr}
	Let $I \lhd R$ be a radical ideal of dimension one. Then $I^{(n)} / I^n = H^0_{\fm}(R/I^n)$.
\end{fact}
Also, we need the following non-trivial result, as it uses some computations from Serre's intersection multiplicity.  
\begin{fact}\label{nonf}
 (See \cite[Corollary 2.5]{h}). Let $S$ be a regular local ring of dimension $3$ and $\fp$ a prime ideal of dimension one which is not a complete intersection. Then $\fp^i \neq \fp^{(i)}$ for all $i > 1$.
\end{fact}

\begin{proposition}\label{6.1d}
	Suppose $S$ is a $3$-dimensional regular ring, $I := P^2$ for some prime ideal $P$, and $R := S/I$. If $\id_R(I/I^2) < \infty$, then $I$ is generated by a regular sequence.
\end{proposition}

\begin{proof}
	First, we show that $P = \fn := (x,y,z)$ is excluded. Suppose, for contradiction, that this occurs. Then $R$ is artinian, and if $\id_R(I/I^2)$ is finite, we must have $\id_R(I/I^2) = 0$. Thus $I/I^2 \cong \bigoplus_r \omega_R$ for some $r \in \mathbb{N}$. Recall that $\operatorname{Soc}(\omega_R)$ is one-dimensional. Since $\operatorname{Soc}(I/I^2) = \fn^3 / \fn^4$ and $\fn^3$ is generated by the $10$ elements $\{x^3, y^3, z^3, x^2y, xy^2, x^2z, xz^2, y^2z, yz^2, xyz\}$, we deduce that $r = 10$. Similarly, $\fn^2$ is generated by the $6$ elements $\{x^2, y^2, z^2, xy, xz, yz\}$. Then we have
	\[
	10 \cdot \ell(\omega_R) = \ell(\fn^2 / \fn^4) = \ell(\fn^2 / \fn^3) + \ell(\fn^3 / \fn^4) = 6 + 10 = 16,
	\]
	a contradiction because $\ell(\omega_R) = 16/10 = 8/5$ is not an integer. Hence $P = \fn$ is excluded.
	
	Next, we assume that $P$ is a height two ideal. By Bass' conjecture, $R$ is Cohen--Macaulay. Thus, by independence theorem for local cohomology modules we have 
	\[
	0 = H^0_{\fm}(R) = H^0_{\fn}(S/P^2) \stackrel{\ref{tr}}{=} P^{(2)} / P^2.
	\]
	
In view of {Fact} \ref{nonf} we see $P = \langle a, b \rangle$ is generated by a regular sequence. Consequently, $P^2 = \langle a^2, b^2, ab \rangle$, and recall that $P^4 \subseteq \fn P^2$. Therefore
	\[
	\mu(P^2 / P^4) = \dim_k\left(\frac{P^2 / P^4}{\fm (P^2 / P^4)}\right) = \dim_k\left(\frac{P^2 / P^4}{(\fn P^2)/ P^4}\right) = \dim_k\left(\frac{P^2}{\fn P^2}\right) = \mu(P^2) = 3,
	\]
	i.e., $P^2 / P^4$ is minimally $3$-generated. Recall that the type of $R$ is two (see \cite[6.6(ii)]{ini}). By Corollary \ref{cor7.7}, $P^2 / P^4$ is minimally generated by an even number of elements. Hence this case is also excluded.
	
	Finally, without loss of generality, we assume that $\operatorname{ht}(P) = 1$. Since $S$ is a UFD, its height-one unmixed ideals are principal. Hence $I = P^2$ is also principal and generated by a regular element.
\end{proof}

Here is the corresponding result for embedding dimension four.

\begin{proposition}\label{6.2d}
	Suppose $S$ is a $4$-dimensional regular ring, $I := P^h$ for some Gorenstein prime ideal $P$ of height $h$, and $R := S/I$. If $\id_R(I/I^2) < \infty$, then $I$ is   generated by a regular sequence.
\end{proposition}

\begin{proof}
	Suppose first that $P = \fn$. So $I := \fn^4$. The minimal number of generators of $\fn^n$ in $d = 4$ variables is $\binom{n+d-1}{d-1} = \binom{n+3}{3}$. Thus:
	\[
	\begin{aligned}
	\mu(\fn^4) &= \binom{7}{3} = 35, \quad \mu(\fn^5) = \binom{8}{3} = 56,\\
	\mu(\fn^6) &= \binom{9}{3} = 84, \quad \mu(\fn^7) = \binom{10}{3} = 120.
	\end{aligned}
	\]
	Since $\operatorname{Soc}(I/I^2) = \fn^7 / \fn^8$ and $\fn^7$ is generated minimally by $120$ elements, we deduce that $I/I^2 \cong \bigoplus_{35} \omega_R$. Then we have
	\[
	120 \cdot \ell(\omega_R) = \ell(\fn^4 / \fn^8) = \ell(\fn^4 / \fn^5) + \ell(\fn^5 / \fn^6) + \ell(\fn^6 / \fn^7) + \ell(\fn^7 / \fn^8) = 35 + 56 + 84 + 120 = 295,
	\]
	a contradiction because $\ell(\omega_R) = 295/120 = 59/24$ is not an integer. Hence $P = \fn$ is excluded.
	
	Now, suppose $P$ is of height three. So $I := P^3$.   Recall that $\type( R)=3$ (see \cite[6.6(ii)]{ini}). In particular, $R$ is not Gorenstein. But by Bass' conjecture, $R$ is Cohen--Macaulay. Thus,
by independence theorem for local cohomology, we have	\[
	0 = H^0_{\fm}(R) = H^0_{\fn}(S/P^3) \stackrel{\ref{tr}}{=} P^{(3)} / P^3.
	\]
	
	In view of \cite[Corollary 2.6]{h}, $P = \langle a, b, c \rangle$ is generated by a regular sequence. Consequently,
	\[
	P^3 = \langle a^3, a^2b, ab^2, a^2c, ac^2, b^3, b^2c, bc^2, abc \rangle,
	\]
	and recall that $P^6 \subseteq \fn P^4$. Therefore
	\[
	\mu(P^3 / P^6) = \dim_k\left(\frac{P^3 / P^6}{\fm (P^3 / P^6)}\right) = \dim_k\left(\frac{P^3 / P^6}{(\fn P^3)/ P^6}\right) = \dim_k\left(\frac{P^3}{\fn P^3}\right) = \mu(P^3) = 10,
	\]
	i.e., $P^3 / P^6$ is minimally $10$-generated. Let $x$ be a system of parameters for $R$. In light of \cite[II.5.4]{PS}, we have
	\[
	\Ext_R^1\bigl(\Hom_R(k, R/xR), P^3 / P^6\bigr) \cong k \otimes_R \frac{P^3 / P^6}{x(P^3 / P^6)} \cong \frac{P^3 / P^6}{\fm (P^3 / P^6)}.
	\]
	Thus $3 \cdot \mu_1(P^3 / P^6) = \mu_1(P^3 / P^6) \cdot \type(R) = 10$, a contradiction. Hence this case is also excluded.
	
	Now, suppose $P$ is of height two. So $I := P^3$. By a result of Serre, $P$ is a complete intersection, as we assumed that $R/P$ is Gorenstein. Suppose for contradiction that $\id_R(I/I^2) < \infty$. Then by a computation similar to that in the previous result, $P^2 / P^4$ is minimally $3$-generated, the type of $R$ is two, and $\mu(P^2 / P^4) \in 2\mathbb{N}$. Hence this case is also excluded.
Finally, without loss of generality, we assume that $\operatorname{ht}(P) = 1$. Since $S$ is a UFD, its height-one unmixed ideals are principal. 
	
	In summary, we proved that \[
	\id_R(I/I^2) < \infty\Longrightarrow   I \emph{ is complete intersection},
	\]as desired.
\end{proof}
These are special case of:
\begin{theorem}\label{thm:5power} Suppose $(S,\mathfrak n)$ is a regular local ring, $P$ is a nonzero prime ideal of height $h$, $q\ge 2$, $I=P^q$ and $R=S/I$.
	If $\operatorname{id}_R(I/I^2)<\infty$, then $h=1$. Conversely, if $h=1$, then $I$ is principal and $I/I^2\cong R$, so its injective dimension over $R$ is finite. Thus
	\[
	\operatorname{id}_{S/P^q}(P^q/P^{2q})<\infty\quad\Longleftrightarrow\quad \operatorname{ht}P=1.
	\]
\end{theorem}
\begin{proof}
	Localize at $P$. Put $T=S_P$, $\mathfrak m=PT$, $k=T/\mathfrak m$, $A=T/\mathfrak m^q$, and $C=\mathfrak m^q/\mathfrak m^{2q}$. The regular local ring $T$ has dimension $h$, and localization preserves finite injective dimension. As $A$ is Artinian, $C\ne0$ has finite injective dimension precisely when $C$ is injective. In that event $C\cong E_A(k)^{\oplus r}$, with $r=\dim_k\operatorname{Soc}_A C$. Since $\operatorname{Soc} E_A(k)$ is one-dimensional, and $E_A(k)$ is a finite $A$-module, one has
	\[
	\mu_A(C)=r\,\mu_A(E_A(k))\ge r=\dim_k\operatorname{Soc}_A C. \tag{$*$}
	\]
	The associated graded ring $\operatorname{gr}_{\mathfrak m}T$ is the polynomial ring $k[X_1,\dots,X_h]$. Hence
	\[
	\mu_A(C)=\dim_k\mathfrak m^q/\mathfrak m^{q+1}
	=\binom{h+q-1}{h-1}.
	\]
	Moreover,
	\[
	\operatorname{Soc}_A(C)=(0:_C\mathfrak m)=\mathfrak m^{2q-1}/\mathfrak m^{2q}.
	\]
	Indeed, if a nonzero element has initial degree $j<2q-1$, multiplication by a suitable (in fact any nonzero initial linear form) raises its initial degree to $j+1<2q$, since the associated graded ring is a domain. Therefore
	\[
	\dim_k\operatorname{Soc}_A C=\binom{h+2q-2}{h-1}.
	\]
	For $h\ge2$ and $q\ge2$, the latter exceeds $\binom{h+q-1}{h-1}$, contradicting $(*)$. Consequently $h=1$. A regular local ring is factorial, so any height-one prime is generated by a nonzerodivisor $f$. Then $I=(f^q)$, $I/I^2\cong S/I=R$, and $R$ is a hypersurface, hence Gorenstein; therefore $\operatorname{id}_R R<\infty$.
\end{proof}

\begin{remark}
Take $S=k[x,y]$, $\mathfrak m=(x,y)$ and $I=\mathfrak m^2=(x^2,xy,y^2)$. For each degree-two monomial $a_i$, $(\mathfrak m^4:a_i)=\mathfrak m^2=I$, so the stated hypothesis holds. Put $A=S/\mathfrak m^2$ and $C=\mathfrak m^2/\mathfrak m^4$. The class $c=xy^2+\mathfrak m^4$ is nonzero, and
\[
c=x\,(y^2+\mathfrak m^4)=y\,(xy+\mathfrak m^4),
\]
so $c\in xC\cap yC$. But $xA\cap yA=0$, hence $xA^n\cap yA^n=0$ for all $n$. For every $A$-linear map $f:C\to A^n$, $f(c)\in xA^n\cap yA^n=0$. Thus no such map is injective.
\end{remark}

\begin{remark}\label{e}
1)	(Jacobian criterion). Let $R$ be an equidimensional affine ring (not necessarily reduced) over a perfect field $k$. The following assertions are true:
	\begin{enumerate}
		\item[(i)] Suppose that $\Omega_{R/k}$ is locally free over $R$ and $R$ is reduced or $\operatorname{char} k = 0$. Then $\Omega_{R/k}$ has rank $d := \dim R$.
		\item[(ii)] The module $\Omega_{R/k}$ is locally free over $R$ of rank $d$ if and only if $R_{\fp}$ is a regular local ring for each prime ideal $\fp$ of $R$.
	\end{enumerate}

2) Now, suppose $S$ is essentially of finite type over $k$.
	\begin{enumerate}
		\item[(i)] We recall the conormal sequence
		\[
		I/I^2 \stackrel{\delta}{\longrightarrow} \Omega_{S/k} \otimes_S R \longrightarrow \Omega_{R/k} \longrightarrow 0.
		\]
		It was claimed somewhere that $\delta$ is injective provided that $I/I^2$ is torsion-free. However, this requires the additional assumption that $I/I^2$ has a rank. If $\delta$ were  be injective and $I$ be $\fm$-primary, then   $I$ is generated by a regular sequence. Indeed, first, recall that the $R$-module $I/I^2$ is injective, because its injective dimension is bounded by the $\depth(R)=\dim(R)=0$. By Matlis' decomposition theorem, $I/I^2 \cong \bigoplus \omega_R$.  Hence, for some $n$, we have
		\[
		\omega_R \stackrel{\Delta}{\hookrightarrow} \bigoplus \omega_R \cong I/I^2 \hookrightarrow R^n,
		\]
		where $\Delta$ is the diagonal map. This shows that the canonical module $\omega_R$ is torsionless. By  \cite[3.7]{moh2}, $R$ is Gorenstein. Over Gorenstein rings, finite injective dimension implies finite projective dimension. Therefore $\pd_R(I/I^2)$ is finite. It follows from \cite{v} that $I$ is generated by a regular sequence, as claimed.
		
		\item[(ii)] Recall that over rings of depth zero, every module is torsion-free. Let $S = k[x,y]$ and $I := (x,y)^2$. In the ring $R := k[x,y]/(x,y)^2$, the element $x^2y + I^2$ is nonzero and lies in the kernel of $\delta : I/I^2 \to \Omega_{S/k} \otimes_S R$, since this map is determined by the assignment $f + I^2 \mapsto (\partial f/\partial x + I,\; \partial f/\partial y + I)$.
		
		\item[(iii)] Adopt the previous notation. Then $I^2 = (x,y)^4$ and
		\[
		(I^2 :_S I) = ((x,y)^4 : (x,y)^2) = (x,y)^2 = I.
		\]
		Hence $\operatorname{Ann}_{S/I}(I/I^2) = \frac{(I^2 :_S I)}{I} = 0$. Therefore $I/I^2$ is a faithful $S/I$-module.
		
		\item[(iv)] Recall that $I$ is not a complete intersection. From this we deduce that $\id(I/I^2) = \infty$, as one may check more directly (see Proposition \ref{6.1d}).
	\end{enumerate}

3) The condition \( (I^2 : a_i)=I \) for  \( a_i\in G(I) \) is rather strong; nevertheless, it is satisfied by several natural classes of monomial ideals. 	Let
	\(
	I=(x,y)\subseteq k[x,y].
	\)
	Then
	\(
	I^2=(x^2,xy,y^2).
	\)
	We compute
	\(
	(I^2:x)
.
	\)
	A monomial $u$ satisfies $ux\in I^2$ if and only if $u$ has positive degree, equivalently $u\in (x,y)$. Hence
	\(
	(I^2:x)=(x,y)=I.
	\)
	Similarly,
	\(
	(I^2:y)=I.
	\)
	Therefore
	\(
	(I^2:a)=I
	\)
	for every minimal generator $a\in G(I)$.  More generally,
let
	\(
	I=(x^a,y^b)\subseteq k[x,y]\),
	\(a,b\ge 1.
	\)
	Then
the hypothesis
	\(
	(I^2:a_i)=I
	\)
	holds for every minimal generator of $I$.
\end{remark}

\begin{proposition}\label{511}		Let $R$ be of dimension at most $1$, and complete. Assume the injective dimension of the normal bundle $N$ is finite. Then $I$ is generated by a regular sequence.  
\end{proposition}

\begin{proof}
First, assume $R$ is zero-dimensional.	Let $(-)^\ast$ denote $\Hom_R(-,R)$. Recall that $N := (I/I^2)^\ast$, where $R = S/I$. Also recall that any module of the form $(-)^\ast$ is torsionless.
Then
there is a free module $F$ so that the following is exact:
$$0\lo (I/I^{2})^\ast\lo F\lo  \frac{F}{ (I/I^{2})^\ast}  \lo 0$$
It follows that $(I/I^{2})^\ast$ is injective. Thus, the sequence
splits.  By the Krull--Remak--Schmidt theorem,
$(I/I^{2})^\ast$ is free. Since $\depth(R)=0$,
$I/I^{2}$ is free (see \cite{ram}). By
Vasconcelos and Ferrand, $I$ is generated by a regular sequence (see \cite{v}).

Now, assume $R$ is 1-dimensional. Recall that $N := (I/I^2)^\ast$  is torsionless. In Lemma \ref{in} (see below) we will show that $R$ is generically Gorenstein, since there exists a nonzero torsionless module of finite injective dimension. Over such a ring, any dual module is reflexive (see e.g. \cite[Fact 5.14]{moh2}). In particular, $N$ is reflexive. In {Proposition} \ref{74} (see below) we will show that $R$ is Gorenstein. Thus, $\pd(N) < \infty$. The ring $R$ is Cohen–Macaulay, i.e., $\depth(R) = 1$. But $\depth(\Hom_R(-,R)) \geq \min\{2,\depth(R)\} = 1 > 0$ (see \cite[1.4.19]{BH}). Combining these with the Auslander–Buchsbaum formula, it follows that $N$ is free. It turns out that $I/I^2$ is free. In this case, by the work of Vasconcelos and Ferrand, $I$ is generated by a regular sequence (see \cite{v}).
\end{proof}

\begin{corollary}\label{511a}	The following assertions are valid:

	\begin{enumerate}
	\item[(i)]  Suppose $R$ is zero-dimensional, and assume $\id_R((I/I^{(2)})^\ast) < \infty$. Then $I$ is generated by a regular sequence.  
	\item[(ii)] Suppose $R$ is a $1$-dimensional complete reduced ring with perfect residue field, and assume $\id_R((I/I^{(2)})^\ast) < \infty$. Then $I$ is generated by a regular sequence.  
\end{enumerate}	
\end{corollary}

\begin{proof}
(i) Since $R$ is zero-dimensional, we deduce that 
$ I^2=I^{(2)}$. So, the desired claim is in {Proposition} \ref{511}.

(ii) Recall that $I/I^2$ and $I/I^{(2)}$ are connected by the short exact sequence:\[
0 \longrightarrow I^{(2)}/I^2 \longrightarrow I/I^2 \longrightarrow I/I^{(2)} \longrightarrow 0 \quad(\ast)
\]	The ideal $I$ is radical since $S/I$ is reduced. Then from the exact sequence
	\[
	0 \longrightarrow I^{(2)}/I^2 \longrightarrow I/I^2 \stackrel{\delta}{\longrightarrow} \Omega_{S/k} \otimes_S R \longrightarrow \Omega_{R/k} \longrightarrow 0,
	\]
	we see that $I/I^2$ and $I/I^{(2)}$ have the same dual because their difference $I^{(2)}/I^2$ is torsion and consequently has trivial dual (see $(\ast)$). Hence the injective dimension of the normal bundle is finite. It remains to apply Proposition \ref{511}.
\end{proof}

\begin{proposition}	The following assertions are valid:
	
	\begin{enumerate}
		\item[(i)]  Over reduced rings, Question~\ref{q51}(i) reduces to locally complete intersection ideals on the punctured spectrum.
		\item[(ii)] Question~\ref{q51}(ii) (resp. (iii)) reduces to locally complete intersection ideals on the punctured spectrum.  
	\end{enumerate}	
	
\end{proposition}

The following proof of (i) works for generically Gorenstein instead of reducedness, but we think that this condition can be removed as well.

\begin{proof}
(i)	We proceed by induction on $d := \dim R$. The case $d = 0$ is clear because the ring is reduced, and so becomes a field by dimension considerations. Thus we may assume $d > 0$ and suppose the desired claim holds for rings of dimension $< d$. Let $\mathfrak{q} \in \operatorname{Var}(I) \setminus \{\mathfrak{m}\}$. Since $\operatorname{id}_R(I/I^2) < \infty$ and localization of an injective module over a Noetherian ring is injective, we have $\operatorname{id}_{R_{\mathfrak{q}}}(I R_{\mathfrak{q}} / I_{\mathfrak{q}}^2) \leq \operatorname{id}_R(I/I^2) < \infty$. Recall that $\dim(R_{\mathfrak{q}}) < d$. These facts allow us to apply the inductive step and assume, in addition, that $I$ is locally a complete intersection over the punctured spectrum. This completes the proof.

(ii) The only difference from part (i) is the case $d = 0$, where $d := \dim R$. In this case, it is enough to apply Corollary \ref{511a}.
\end{proof}

In Proposition \ref{6.1d}, we considered ideals of the form $\mu(I) = \operatorname{ht}(I) + 1$. Here is another situation for which $\mu(I) \leq \operatorname{ht}(I) + 1$.

\begin{observation}\label{alm}
	Suppose $I$ is prime, $\mu(I) \leq \operatorname{ht}(I) + 1$, and $\id_R(I/I^2) < \infty$. Then $I$ is generated by a regular sequence.
\end{observation}

\begin{proof}
	By Bass' conjecture, $R$ is Cohen--Macaulay. Then $K_R = \omega_R$ has finite injective dimension. Suppose, for contradiction, that $\mu(I) = \operatorname{ht}(I) + 1$. In light of \cite[Proposition 1]{matsuoka}, there exists an exact sequence
	\(
	0 \longrightarrow K_R \longrightarrow R^n \longrightarrow I/I^2 \longrightarrow 0,
	\)
	which shows that $\id_R(R) < \infty$. Over Gorenstein rings, finite injective dimension implies finite projective dimension. Hence $\pd_R(I/I^2)$ is finite. It follows from \cite{v} that $I$ is generated by a regular sequence, as claimed.
\end{proof}

Let us present a new proof of \cite{kunz}:

\begin{theorem}[Kunz]
	Almost complete intersection domains are not Gorenstein.
\end{theorem}

\begin{proof}
	We adopt the previous notation. Suppose, for contradiction, that $R$ is Gorenstein. From $0 \to K_R \to R^n \to I/I^2 \to 0$ and the finiteness of $\id(R)$, we conclude that $\id_R(I/I^2) < \infty$. By the previous result, $\operatorname{ht}(I) = \mu(I)$, a contradiction.
\end{proof}

\begin{corollary}\label{513}
	Suppose $I$ is an unmixed ideal and generated by $3$ elements. If $\id_R(I/I^2) < \infty$, then $I$ is generated by a regular sequence.
\end{corollary}

\begin{proof}
	We have $\operatorname{ht}(I) \leq 3$. If $\operatorname{ht}(I) = 1$, there is nothing to prove because $I$ becomes principal as $S$ is a UFD. If $\operatorname{ht}(I) = 3$, we use \cite[1.2.21]{BH}, which says that $I$ is generated by a regular sequence, and the claim follows. Thus we may assume $\operatorname{ht}(I) = 2$. The desired claim then follows from Observation \ref{alm}\footnote{Gheibi informed us that, in a joint work with Takahashi \cite{G}, they have recently generalized Observation 5.12 from prime ideals to arbitrary ideals.}.
\end{proof}

In some how sense, the class of Cohen--Macaulay rings with minimal multiplicity is opposite to the class of Gorenstein rings. Also, the normal module $N_I$ is a  second syzygy over Cohen--Macaulay rings. This suggests:

\begin{observation}\label{m517}
	Let $R$ be a $2$-dimensional Cohen--Macaulay ring with minimal multiplicity, and assume that the normal module $N_I$ is a third syzygy. One has
	\[
	\id_R(N_I)<\infty\Longrightarrow \text{  I is generated by a regular sequence}.\]
	It maybe $R$ is not regular, if we assume $N_I$ is free.
\end{observation}

\begin{proof}
	Suppose first that $N_I$ is free. Then $I/I^2$ is free. It turns out that $I$ is free, and so principal, and   generated by a regular element.
We may assume that $N_I$ is not free. Without loss of generality we may assume that $k$ is infinite. 
	Let $F$ be a free $R$-module such that
	\[
	\zeta:=\qquad
	0 \longrightarrow N_I
	\longrightarrow F
	\longrightarrow \Omega^{-1}(N_I)
	\longrightarrow 0
	\]
	is exact. 
	Since $N_I$ is  third syzygy module, $\Omega^{-1}(N_I)$ is torsionless, that is submodule of a free module. Hence, for every parameter element $r\in R$,
	\[
	\Tor_1^R(R/rR,\Omega^{-1}(N_I))
	=
	\ker\!\left(
	\Omega^{-1}(N_I)
	\xrightarrow{\;r\;}
	\Omega^{-1}(N_I)
	\right)
	=0.
	\]
	
	Therefore, tensoring $\zeta$ with $\overline{R}=R/rR$ yields an exact sequence
	\[
	\overline{\zeta}:=\qquad
	0\longrightarrow
	\overline{N_I}
	\longrightarrow
	\overline{F}
	\longrightarrow
	\overline{\Omega^{-1}(N_I)}
	\longrightarrow 0.
	\]
	
	Let $\underline{x}=x_1, x_2$ be a system of parameters such that
	\(
	\mathfrak m^2=\underline{x}\,\mathfrak m.
	\)
	Set
	\(
	\overline{R}=R/(\underline{x})R.
	\)
	By repeatedly applying the above argument, the sequence
	\(
	\zeta\otimes_R \overline{R}
	\)
	remains exact.
	In particular, $\overline{N_I}$ is a syzygy module over $\overline{R}$, and hence
	\(
	\overline{N_I}\subseteq
	\mathfrak m_{\overline{R}}\,
	\overline{R}^{\,n}
	\)
	for some integer $n$. Since $\mathfrak m_{\overline{R}}^2=0$, it follows that
	\(
	\mathfrak m_{\overline{R}}\overline{N_I}=0.
	\)
	Thus $\overline{N_I}$ is a vector space, i.e., a 
	direct sum of residue field
	\(
	k=\overline{R}/\mathfrak m_{\overline{R}}.
	\) But $\depth(N_I)=\depth(\Hom_R(I/I^2,R)) \geq \min\{2,\depth(R)\} =2$ (see \cite[1.4.19]{BH}). So, it is maximal Cohen-Macaulay. In particular, $\underline{x}$ is both $R$-regular and $N_I$-regular sequence.
	Then, by \cite[Proposition 3.1.15]{BH},
	\[
	\id_{\overline{R}}
	\!\left(
	\bigoplus k
	\right)=\id_{\overline{R}}(\overline{N_I})
	=
	\id_R(N_I)-2
	<\infty.
	\]
	Consequently,
	\(
	\id_{\overline{R}}(k)<\infty,
	\)
	and therefore $\overline{R}$ is a field.
	Hence
	\[
	\mu(\mathfrak m)
	=
	d+\mu(\mathfrak m_{\overline{R}})
	=
	d
	=
	\dim R.
	\]
	Therefore $R$ is regular. So, I is generated by a regular sequence.

	To see the last assertion, take $S=k[[x,y,z]]$, $I=(xy-z^2)$, and $R=S/I$. Then $R$ is a two-dimensional singular Cohen--Macaulay hypersurface of multiplicity two, and $e(R)=\operatorname{embdim}(R)-\dim R+1=2$, so it has minimal multiplicity. As $I$ is principal, $I/I^2\cong R$ and $N_I\cong R$. This free module is a third syzygy (for instance, a syzygy in a split exact complex of free modules). Moreover $R$ is Gorenstein, so $\operatorname{id}_R N_I=\operatorname{id}_R R=2<\infty$, although $R$ is not regular.
\end{proof}

\begin{proposition}\label{prop:curve}
	Let $S=k[x,y,z]$ (or its localization at $(x,y,z)$), let $P$ be the prime defining the parametrized curve $k[t^a,t^b,t^c]$, where $\gcd(a,b,c)=1$, and let $R=S/P$ (with the corresponding localization if desired). Suppose $R$ is not Gorenstein and $\mu(P)<4$ (in particular, when $P$ has three generators). Then
	\[\operatorname{id}_R(P/P^{(2)})=\infty.\]
	More generally, if $R$ is one-dimensional Cohen--Macaulay with canonical module, is not Gorenstein, and $P/P^{(2)}$ has generic rank two and at most three generators, the same conclusion holds.
\end{proposition}
\begin{proof}
	Set $M=P/P^{(2)}$. At the generic point of the curve, the regular local ring $S_P$ has dimension two, so $M\otimes_R\operatorname{Frac}(R)$ has dimension two over the fraction field: indeed $P^{(2)}S_P=P^2S_P$ and $PS_P/P^2S_P$ is a two-dimensional vector space. Thus $\operatorname{rank}_R M=2$, not three. The natural map from $M$ to the free module $(S/P)^3$ is not needed for this computation.
	Suppose $\operatorname{id}_R M<\infty$. Bass' theorem implies $R$ is Cohen--Macaulay (already true for this curve). As $R$ is a one-dimensional domain, $M$ is torsion-free: it embeds in $(P/P^2)\otimes_R\operatorname{Frac}(R)$, because $P^{(2)}$ is precisely the kernel of the localization map $P\to P S_P/P^2S_P$. Hence $M$ is maximal Cohen--Macaulay. Passing to the completion if necessary, the structure theorem for maximal Cohen--Macaulay modules of finite injective dimension gives $\widehat M\cong\omega_{\widehat R}^{\oplus 2}$, where the multiplicity is the generic rank. Since $R$ is not Gorenstein, $\mu(\omega_R)=\operatorname{type}(R)\ge2$. Therefore
	\[\mu_R(M)=2\operatorname{type}(R)\ge4,\]
	whereas $M$ is a quotient of $P/P^2$ and thus $\mu_R(M)\le\mu_S(P)<4$, a contradiction.
\end{proof}

\begin{remark}
	If $R$ is Gorenstein and $\id_R(I/I^2) < \infty$, then $I$ is generated by a regular sequence. Indeed, since $R$ is Gorenstein, $\pd_R(I/I^2) < \infty$. In view of \cite[main result]{b}, $I$ is generated by a regular sequence.
\end{remark}

One may consider other operations such as $\overline{I^{2}}$ the integral closure of $I^{2}$. Namely,  when is $\id_R(I/\overline{I^{2}}) < \infty$? Here, one needs more care. For example if $S$ is two-dimensional, it follows by Brian\c{c}on--Skoda Theorem that $\overline{I^{2}} \subseteq I$. Also, when $I$ is radical, one has 
$\overline{I^{2}} \subseteq I$.
\begin{example}
Let $S = k[[x,y]]$ be the formal power series ring over a field $k$, and let $\mathfrak{m} = (x,y)$ be the maximal ideal. Consider the ideal $I = (x^2, xy)$.
Then
$\id_R(I/\overline{I^{2}}) = \infty$.
\end{example}

\begin{proof}First, we show that
	$I / \overline{I^2}$ is a vector space over $k$:
\begin{itemize}
	\item $I^2 = (x^4, x^3y, x^2y^2)$,
	\item $\overline{I^2}$ contains $x^3$ and $x^2y$ because $(x^3)^2 = x^6 \in I^2$ and $(x^2y)^2 = x^4y^2 \in I^2$. Also, $xy^2\in\overline{I^2}$,
	\item $\mathfrak{n} I = (x,y)(x^2, xy) = (x^3, x^2y, xy^2) \subseteq \overline{I^2}$,
\end{itemize}

Thus $I / \overline{I^2}$ is a vector space over $k = S/\mathfrak{n}=R/\mathfrak{m}$. If $\id_R(I/\overline{I^{2}}) <\infty$, it follows that $R$ is regular. But, $R$ is not even an integral domain.
\end{proof}

\begin{lemma}\label{lem:mpowers}
	Let $(T,\m,k)$ be regular local of dimension $h\ge2$ and $r\ge2$. Then
$\id_{T/\m^r}(\m^r/\m^{2r})=\infty.$
\end{lemma}
\begin{proof}
	Apply Lemma~\ref{lem:param} with the parameter ideal $Q=\m$ and $q=r$. Equivalently, the generator count is $\binom{h+r-1}{h-1}$ and the socle dimension is $\binom{h+2r-2}{h-1}$, a strictly larger number.
\end{proof}
\begin{theorem} \label{thm:symbolic}
	Let $(S,\n)$ be a regular local ring, let $P$ be a nonzero prime, and let $r\ge2$. Define $I=P^{(r)}=P^rS_P\cap S$. Then
	\[\id_{S/I}(I/I^2)<\infty
		\quad\Longleftrightarrow\quad \operatorname{ht}P=1.\quad 
	\]
	If the conditions hold, $I=(f^r)$ for a prime element $f$ and is a complete intersection.
\end{theorem}
\begin{proof}
	Localization at $P$ yields $IS_P=(PS_P)^r$, and thus
	\[
	(S/I)_P=S_P/(PS_P)^r,\qquad
	(I/I^2)_P=(PS_P)^r/(PS_P)^{2r}.
	\]
	If $\operatorname{ht}P\ge2$, Lemma~\ref{lem:mpowers} contradicts the assumed finite injective dimension. Therefore $\operatorname{ht}P=1$. As $S$ is factorial, $P=(f)$; all symbolic and ordinary powers of this height-one prime agree. The converse is the hypersurface case.
\end{proof}
\begin{corollary} \label{cor:genericmax}
	Let $0\ne I\subsetneq S$ be an ideal. Assume that for each $P\in\operatorname{Min}(S/I)$ there is an integer $r(P)\ge2$ such that
$IS_P=(PS_P)^{r(P)}.$
	If $\id_{S/I}(I/I^2)<\infty$, then $I$ is principal.
\end{corollary}
\begin{proof}
	Lemma~\ref{lem:mpowers} excludes every minimal prime of height at least two. Thus $\operatorname{ht}I=1$. The quotient is Cohen--Macaulay by the Bass conjecture, so $\pd_S(S/I)=1$ and $I$ is free of rank one by the argument in Theorem~\ref{thm:powers}.
\end{proof}

\begin{proposition}\label{prop:layers}
	Let $(R,\m,k)$ be a Cohen--Macaulay local ring with canonical module $\omega_R$, and let $0\ne M$ be a finite $R$-module of finite injective dimension. Write $t=\mu_R(\omega_R)=\type(R)$. Then
	\[
	t\mid\mu_R(M),\qquad
	a:=\frac{\mu_R(M)}t\in\mathbb Z_{>0},
	\]
	and there is an epimorphism $\omega_R^{\oplus a}\twoheadrightarrow M$ inducing an isomorphism on reduction modulo $\m$. Consequently, 
$\quad\length_R(M/\m^jM)\ \leq\
			a\,\length_R(\omega_R/\m^j\omega_R)$ for every $j\ge1$.
\end{proposition}
\begin{proof}
We are going to use the Foxby equivalence. To this end let $N:=\operatorname{Hom}_R(\omega_R,M)$. Then $N$ has finite projective dimension and the evaluation map $\omega_R\otimes_R N\to M$ is an isomorphism.
Choose a minimal free cover $R^{\oplus a}\twoheadrightarrow N$, where $a=\mu_R(N)$. Tensoring with $\omega_R$ gives a surjection $\omega_R^{\oplus a}\twoheadrightarrow M$. Moreover,
	\[
	\mu_R(M)=\dim_k(M/\m M)
	=\dim_k((\omega_R/\m\omega_R)\otimes_k(N/\m N))
	=t\,a.
	\]
	The cover is therefore minimal. Reduction of the surjection modulo $\m^j$ remains surjective, and the length inequality follows.
\end{proof}
\begin{corollary}\label{cor:quadratic}
	Let $S=k[Z_1,\ldots,Z_N]$ be standard graded, let $I\subseteq S$ be a homogeneous ideal generated by quadrics, and assume $R=(S/I)_{\m}$ is Cohen--Macaulay at the homogeneous maximal ideal. Suppose the graded canonical module $\omega$ of $S/I$ is generated in a single degree $b$. Let
	\[
	t=\dim_k\omega_b,\qquad
	v=\dim_k\omega_{b+1},\qquad
	u=\dim_k I_2,\qquad w=\dim_k I_3.
	\]
	If $\id_R((I/I^2)_{\m})<\infty$, then
$	t\mid u,$ and $t\,w\leq u\,v.$
\end{corollary}
\begin{proof}
	The conormal module $C=I/I^2$ is generated in degree two. Because $I^2$ has no elements in degree three, $C_2=I_2$ and $C_3=I_3$. Thus
$\length_R(C_{\m}/\m^2C_{\m})=u+w.$ As $\omega$ is generated in degree $b$,
	$\length_R(\omega_{\m}/\m^2\omega_{\m})=t+v$.
	Apply Proposition~\ref{prop:layers} with $a=u/t$ and $j=2$; cancellation of $u$ gives $w\leq (u/t)v$.
\end{proof}
\begin{remark}
	This obstruction is strictly stronger than the divisibility $\type(R)\mid\mu(I)$. It does not require the canonical cover to be homogeneous: the comparison is made after passage to the $\m$-adic quotients, so an \emph{ungraded} local surjection suffices. The bound is a necessary condition, not a sufficient condition for finite injective dimension.
\end{remark}

\begin{notation}
Let $B=k[x_1,\ldots,x_n]$ with its usual grading and
\[
V_n=B^{(2)}=\bigoplus_{j\ge0}B_{2j},\qquad
N=\binom{n+1}{2}.
\]
Set $S=k[Z_{ij}:1\le i\le j\le n]$, with $\deg Z_{ij}=1$, and let
\[
\varphi:S\longrightarrow V_n,\qquad Z_{ij}\longmapsto x_ix_j,
\qquad I_n=\ker\varphi.
\]
We localize $S$ and $V_n$ at their homogeneous maximal ideals as necessary, and denote the localized quotient by $R_n$.
\end{notation}

\begin{lemma}\label{lem:veronese}
	For $n\ge2$ the ideal $I_n$ is generated by quadrics, and
	\[
	\begin{split}
		\mu(I_n)&=\binom{N+1}{2}-\binom{n+3}{4}
		=\frac{n^2(n^2-1)}{12},\\
		\dim_k(I_n/I_n^2)_3&=\binom{N+2}{3}-\binom{n+5}{6}
		=\frac{n^2(n^2-1)(7n^2+15n+32)}{360}.
	\end{split}
	\]
	If $n$ is odd, the canonical module of $V_n$ is generated in degree $b=(n+1)/2$ and
	\[
	\dim_k(\omega_{V_n})_b=n,
	\qquad\dim_k(\omega_{V_n})_{b+1}=\binom{n+2}{3}.
	\]
\end{lemma}
\begin{proof}
	The kernel of $\varphi$ is generated by the quadratic ``pair-exchange'' relations
	$Z_{ij}Z_{kl}-Z_{ik}Z_{jl}$ and their variants with the indices rearranged. Indeed, two products of $Z$-variables have the same image exactly when they represent two pairings of the same multiset of $2j$ indices, and every pair partition can be transformed into another by successive exchanges of partners. Hence the ideal is generated in degree two. There are no linear relations, and $S_m$ has dimension $\binom{N+m-1}{m}$ while $(V_n)_m=B_{2m}$ has dimension $\binom{n+2m-1}{2m}$. Subtraction in degrees two and three gives the first displayed formulas, since $(I_n)^2$ begins in degree four. The polynomial expressions are obtained by expanding the binomial coefficients.
	
	The canonical module of the Veronese ring is the corresponding Veronese submodule of $\omega_B=B(-n)$, namely 
$(\omega_{V_n})_j\cong B_{2j-n}.$
	For odd $n$, the first nonzero component occurs at $j=b=(n+1)/2$ and equals $B_1$, with dimension $n$; the next is $B_3$, with dimension $\binom{n+2}{3}$. Multiplication by $(V_n)_1=B_2$ generates every later odd-degree polynomial from $B_1$, so this first component generates the entire canonical module.
\end{proof}

\begin{theorem}[Second Veronese rings]\label{thm:veronese}
	For every field $k$ and integer $n\ge3$,
$\id_{R_n}\bigl((I_n/I_n^2)_{\m}\bigr)=\infty.$
\end{theorem}
\begin{proof}
	Assume first that $n\ge3$ is odd. The ring $R_n$ is Cohen--Macaulay (the Veronese subring of a polynomial ring), and by Lemma~\ref{lem:veronese} it has type $t=n$. If its conormal module had finite injective dimension, Corollary~\ref{cor:quadratic} would give
	\[
	n\dim_k(I_n/I_n^2)_3
	\leq \mu(I_n)\binom{n+2}{3}.
	\]
	But direct substitution of the formulas from Lemma~\ref{lem:veronese} gives
$$n\dim_k(I_n/I_n^2)_3
			-\mu(I_n)\binom{n+2}{3} 
			\hspace{10mm}=\frac{n^3(n-1)(n+1)(n^2+11)}{180}>0.
$$
	This is a contradiction.
	
	Now let $n\ge4$ be even. The displayed canonical-module formula gives
	$\omega_{V_n}\cong V_n(-n/2)$, so $R_n$ is Gorenstein. If its conormal module had finite injective dimension, then it would also have finite projective dimension over $R_n$; Briggs's theorem~\cite{b} would force $I_n$ to be a complete-intersection ideal. This is impossible: $\operatorname{ht}I_n=N-n=n(n-1)/2$, whereas
	\[
	\mu(I_n)=\frac{n^2(n^2-1)}{12}>
	\frac{n(n-1)}2\quad(n\ge3).
	\]
	Finally, for $n=2$, $I_2$ is the principal quadratic ideal $(Z_{11}Z_{22}-Z_{12}^2)$, so the conormal module is free of rank one over the hypersurface quotient, which is Gorenstein.
\end{proof}

For the Veronese applications normality and Cohen--Macaulayness are automatic.

\begin{proposition} 
	\label{prop:general-veronese}
	Let $d,n\geq2$, let $B=k[x_1,\ldots,x_n]$, and let
	\[
	V_{n,d}=\bigoplus_{j\geq0} B_{dj},\quad
	N=\binom{n+d-1}{d},\quad
	S=k[Z_1,\ldots,Z_N],\quad I=\ker(S\twoheadrightarrow V_{n,d}).
	\]
	Let $e$ be the unique integer with $0\leq e<d$ and $e\equiv-n\pmod d$. Define
	\begin{align*}
		u&=\binom{N+1}{2}-\binom{n+2d-1}{2d}, &
		w&=\binom{N+2}{3}-\binom{n+3d-1}{3d},\\
		t&=\binom{n+e-1}{e}, &
		v&=\binom{n+e+d-1}{e+d}.
	\end{align*}
	If the conormal module of $I$ has finite injective dimension after localization at the homogeneous maximal ideal, then
	\[
{\qquad t\mid u,\qquad tw\leq uv.\qquad} \tag{V}
	\]
\end{proposition}
\begin{proof}
	The defining ideal of a Veronese subring is generated by quadratic relations. One direct argument is to represent a degree-$j$ monomial of $S$ by a partition of $dj$ indices into $j$ groups of size $d$. Any two such partitions with the same underlying multiset can be connected by successive exchanges of one index between two groups; each exchange is a quadratic binomial relation. Therefore $I$ is generated in degree two. The graded dimensions of $S/I$ are $\dim_k(V_{n,d})_j=\binom{n+dj-1}{dj}$, giving $\dim_kI_2=u$ and $\dim_kI_3=w$. As $I^2$ starts in degree four, these are exactly the first two conormal degrees.
	
	The graded canonical module satisfies
$(\omega_{V_{n,d}})_j\cong B_{dj-n}.$
	Its smallest nonzero degree is $b=(n+e)/d$; its first two homogeneous dimensions are $t$ and $v$. The degree-$b$ component generates the canonical module: every monomial of degree $e+ad$ for $a\geq0$ factors as a monomial of degree $e$ times $a$ monomials of degree $d$. The assertion follows from Corollary~\ref{cor:quadratic}.
\end{proof}

\begin{theorem}
	\label{thm:cubic}
	Let $k$ be any field and let $n\geq2$. Denote by $I_{n,3}$ the defining ideal of the cubic Veronese ring $k[x_1,\ldots,x_n]^{(3)}$ in its polynomial presentation. Then
$	\id_{(S/I_{n,3})_{\mathfrak m}}
	\bigl((I_{n,3}/I_{n,3}^{2})_{\mathfrak m}\bigr)=\infty.
$
\end{theorem}
\begin{proof}
	Apply Proposition~\ref{prop:general-veronese} with $d=3$ and $N=\binom{n+2}{3}$. There are three cases.
	
	If $n\equiv2\pmod3$, then $e=1$. Direct expansion and factorization of the binomial expressions give
$$
		tw-uv=\frac{n^3(n-1)(n+1)(n+2)(n+3)
			(2n^3-n^2-5n+82)}{8064}>0.
$$
	Indeed $2n^3-n^2-5n+82=n(2n^2-n-5)+82>0$ for $n\geq2$.
	
	If $n\equiv1\pmod3$, then $e=2$, and the corresponding factorization is
$$
		tw-uv=\frac{n^3(n-1)(n+1)^2(n+2)(n+3)
			(113n^3+227n^2-62n+1672)}{403200}>0.
$$
	For $n\geq2$ the last factor is positive, since $113n^3+227n^2-62n+1672>0$.
	Both cases contradict condition (V).
	
	Finally, if $3\mid n$, the displayed formula for the graded canonical module makes $V_{n,3}$ Gorenstein. Finite injective dimension of the conormal module would therefore imply finite projective dimension; Briggs's theorem forces $I_{n,3}$ to be a complete intersection. It is not: for $n\geq3$ its height is $N-n$ while its number of minimal quadratic generators is $u>N-n$. Indeed, substitution of $N=\binom{n+2}{3}$ and expansion gives
	\[ u-(N-n)=\frac{n(n-1)(3n^4+18n^3+33n^2-22n-160)}{240}>0\quad(n\geq3).
	\] (The only remaining possible case $n=2$ is not divisible by three.)
\end{proof}

\begin{example}
	In degree $3$ and with $n=4$, one has
	\[
	N=20,\quad \operatorname{ht}(I_{4,3})=16,\quad
	u=\binom{21}{2}-\binom{9}{6}=126,
	\]
	\[
	w=\binom{22}{3}-\binom{12}{9}=1320,\quad
	e=2,\quad t=\binom{5}{2}=10,\quad v=\binom{8}{5}=56.
	\]
	Although $t\mid u$ fails already ($10\nmid126$), the second-layer inequality independently fails: $tw=13200>7056=uv$. This is a useful demonstration that different necessary tests can be applied independently; it is not a counterexample.
\end{example}

\medskip
\section{(Semi)-injectivity of total quotient rings}

A ring $R$ of dimension $d := \dim(R)$ is called quasi-Gorenstein if $H^d_{\fm}(R) \cong E_R(k)$. When is $H^d_{\fm}(R) \twoheadrightarrow E_R(k)$ surjective? The following may be regarded as a higher version of \cite[15.17]{E}:

\begin{proposition}
	Suppose $R$ is complete, generically Gorenstein, and Cohen--Macaulay. Then $H^{\dim R}_{\fm}(R) \twoheadrightarrow E_R(k)$ is surjective.
\end{proposition}

\begin{proof}
	Let $d := \dim(R)$. Since the ring is complete and Cohen--Macaulay, $\omega_R$ exists. Because $R$ is generically Gorenstein, $\omega_R \lhd R$. Suppose first that $d = 0$. Then $R$ is Gorenstein, so $\omega_R = E_R(k) = R$. To see that $H^d_{\fm}(R) \twoheadrightarrow E_R(k)$, it remains to note that $H^d_{\fm}(R) = R$. Hence we may assume $d > 0$. Either $\omega_R = R$, or $\omega_R \lhd R$ is of height one. Without loss of generality, we may assume that $\omega_R \lhd R$ is of height one. The Cohen--Macaulay assumption gives $1 = \operatorname{ht}(\omega_R) = \operatorname{grade}(\omega_R, R)$. Let $x \in \omega_R$ be a regular element. This yields an embedding $R \hookrightarrow Rx \subseteq \omega_R$. Applying $\Hom_R(-, E_R(k))$ and using local duality, we obtain the commutative diagram
	\[
	\begin{CD}
	\omega^{\vee} @>>> R^{\vee} @>>> 0 \\
	@A{\cong}AA @A{\cong}AA \\
	H^d_{\fm}(R) @>>> E_R(k)
	\end{CD}
	\]
	In other words, the induced map $H^d_{\fm}(R) \twoheadrightarrow E_R(k)$ is surjective, as claimed.
\end{proof}

\begin{proposition}
	Let $d := \dim(R)$. Suppose $H^d_{\fm}(R) \twoheadrightarrow E_R(k)$. If $d = 0$, then $R$ is Gorenstein.
\end{proposition}

\begin{proof}
	Recall that
	\[
	\begin{CD}
	H^d_{\fm}(R) @>>> E_R(k) @>>> 0 \\
	@V{=}VV @V{=}VV \\
	R @>f>> E_R(k)
	\end{CD}
	\]
	Let $I := \ker(f)$. Then $R/I \cong E_R(k)$. Recall that $R \subseteq E(R) = \bigoplus E_R(k) = \bigoplus (R/I)$. Hence $I = 0$. Thus $R \cong E_R(k)$. Consequently, $R$ is Gorenstein.
\end{proof}

Let $R$ be a $1$-dimensional integral domain and let $Q$ be its fraction field. Matlis proved that $Q / \omega_R \cong E_R(k)$. Here is the higher version (recall that $Q/\omega_R = H^1_{\fm}(\omega_R)$):

\begin{proposition}
	Suppose $R$ is complete, an integral domain, and Cohen--Macaulay. Let $I$ be a nonzero Cohen--Macaulay ideal. Let $d := \dim(R)$. The following are equivalent:
	\begin{enumerate}
		\item[(i)] $\id_R(I) < \infty$,
		\item[(ii)] $I \cong \omega_R$,
		\item[(iii)] $H^d_{\fm}(I) \cong E_R(k)$.
	\end{enumerate}
\end{proposition}

\begin{proof}
	\text{(i) $\Rightarrow$ (ii):} Since $\dim I = d$, the ideal $I$ is maximal Cohen--Macaulay. The desired claim is in \cite[3.3.28]{BH}.
	
	\text{(ii) $\Rightarrow$ (iii):} By right exactness, $H^d_{\fm}(\omega_R) = H^d_{\fm}(R) \otimes_R \omega_R$. Recall that $H^d_{\fm}(\omega_R)$ is artinian. Applying Matlis duality, we obtain
	\[
	\begin{array}{rcl}
	H^d_{\fm}(\omega_R) = H^d_{\fm}(\omega_R)^{\vee\vee}
	& = & \Hom_R\bigl(H^d_{\fm}(R) \otimes_R \omega_R,\; E_R(k)\bigr)^{\vee} \\
	& = & \Hom_R\bigl(\omega_R,\; \Hom_R(H^d_{\fm}(R), E_R(k))\bigr)^{\vee} \\
	& = & \Hom_R(\omega_R, \omega_R)^{\vee} = R^{\vee} = E_R(k).
	\end{array}
	\]
	
	\text{(iii) $\Rightarrow$ (i):} Let $D(-) := \Hom_R(-, \omega_R)$. Some authors also denote it by $I^\dagger$. Applying local duality together with the assumption yields $D(I) = H^d_{\fm}(I)^{\vee} \cong E_R(k)^{\vee} \cong \widehat{R} \cong R$. Over a maximal Cohen--Macaulay modules over Cohen--Macaulay rings, one has $D^2(-) \cong \operatorname{Id}(-)$. Hence
	\[
	I \cong D^2(I) \cong D(R) \cong \omega_R.
	\]
	Since $\omega_R$ has finite injective dimension, $I$ also has finite injective dimension.
\end{proof}

\begin{corollary}
	Let $R$ be a $1$-dimensional integral domain and let $Q$ be its fraction field. Let $I \lhd R$ be nonzero. Then $Q / I \cong E_R(k)$ if and only if $I \cong \omega_R$.
\end{corollary}

The following was proved by Matlis in the $1$-dimensional complete Gorenstein domain case, and by Auslander \cite{comment} in the $\fm$-adic complete case. Since completeness in the $\fm$-adic topology implies completeness in the $R$-topology, this extends both:

\begin{observation}
	Let $R$ be a quasi-local domain complete in the $R$-topology. Then $\Ext^1_R(Q/R, R) \cong R$.
\end{observation}

\begin{proof}
	Completeness in the $R$-topology implies $\Ext^1_R(Q, R) = 0$. Also, $\Hom_R(Q, R) = 0$. The sequence $0 \to R \to Q \to Q/R \to 0$ gives
	\[
	0 = \Hom_R(Q, R) \to \Hom_R(R, R) \to \Ext^1_R(Q/R, R) \to \Ext^1_R(Q, R) = 0,
	\]
	i.e., the claim follows.
\end{proof}

Noetherian local rings are reduced\footnote{Sorry, we use ``reduced'' for different concepts.} by Krull's intersection theorem. So the following is a converse to Auslander \cite{comment}:

\begin{observation}\label{CON}
	Let $R$ be a quasi-local reduced domain such that $\Ext^1_R(Q/R, R) \cong R$. Then $R$ is complete in the $R$-topology.
\end{observation}

\begin{proof}
	The sequence $0 \to R \to Q \to Q/R \to 0$ implies
	\[
	0 = \Hom_R(Q, R) \to \Hom_R(R, R) \to \Ext^1_R(Q/R, R) \to \Ext^1_R(Q, R) \to \Ext^1_R(R, R) = 0.
	\]
	Recall that $\Ext^1_R(Q, R) = \bigoplus_I Q$ is a $Q$-vector space. Localizing the above sequence and using the assumption $\Ext^1_R(Q/R, R) \cong R$, we deduce that $\Ext^1_R(Q, R) = \bigoplus_I Q = 0$. The reduced assumption implies that the sequence
	\[
	0 \to R \to \prod_{r \in R} R/rR \to \Ext^1_R(Q, R) \to 0
	\]
	is exact. Since $\Ext^1_R(Q, R) = 0$, we have $R \cong \prod_{r \in R} R/rR$, as claimed.
\end{proof}

\medskip
\section{Reflexivity and injective dimension}

We start by recalling the following result:

\begin{fact}\label{fullau}
	(Auslander; see \cite[I. Proposition 4.14]{PS}). Suppose $M$ has finite projective dimension and there exists $\fp \in \Supp(M) \cap \Ass(R)$. Then $\Supp(M) = \Spec(R)$.
\end{fact}

\begin{lemma}\label{sfs}
	Suppose $R$ is complete and there exists a torsionless module $M$ of finite injective dimension. Then $\Supp(M) = \Spec(R)$.
\end{lemma}

\begin{proof}
	Let $\tilde{M} := \Hom_R(\omega_R, M)$. Then
	\[
	\Ass_R(\tilde{M}) = \Supp(\omega_R) \cap \Ass(M) = \Spec(R) \cap \Ass(M) = \Ass(M).
	\]
	Since $M$ is torsionless, $M \subseteq R^n$ and therefore $\Ass_R(\tilde{M}) = \Ass(M) \subseteq \Ass(R)$. Recall that $\tilde{M}$ has finite projective dimension. By Fact \ref{fullau}, we have $\Supp(\tilde{M}) = \Spec(R)$. Consequently, $\Ass(R) = \Ass_R(\tilde{M}) = \Ass(M) \subseteq \Ass(R)$. The claim follows.
\end{proof}

The following extends Corollary \ref{pgen}:

\begin{lemma}\label{in}
	Suppose there exists a nonzero torsionless module $M$ of finite injective dimension. Then $\widehat{R}$ is generically Gorenstein.
\end{lemma}

\begin{proof}
	We may assume that $R$ is complete. The ring $R$ is unmixed. Let $\fp \in \Ass(R) = \min(R)$. By Lemma \ref{sfs}, we have $\fp \in \Supp(M)$. Hence $M_{\fp}$ is a nonzero module of finite injective dimension over the artinian ring $R_{\fp}$. Since $\id(-)$ is bounded by depth, we deduce that $M_{\fp} \cong \bigoplus_{n>0} \omega_{R_{\fp}}$. Localization of a torsionless module is again torsionless. It follows that the canonical module of $R_{\fp}$ is torsionless. This implies that $R_{\fp}$ is Gorenstein (see \cite[Corollary 3.7]{moh2}).
\end{proof}

The following is our first main result of this section.

\begin{proposition}\label{74}
	Suppose there exists a nonzero reflexive module $L$ of finite injective dimension. Then $\widehat{R}$ is quasi-normal.
\end{proposition}

\begin{proof}
	We may assume that $R$ is complete. Let $\fp \in \Spec(R)$ be of height one. By Lemma \ref{in}, $\fp \in \Supp(L)$. Set $A := R_{\fp}$ and $M := L_{\fp}$. Then $\id_A(M) = \depth_A(A) = 1$. Consequently, $\Ext_R^2(\Tr(M^\ast), (M^\ast)^\ast) = 0$, where $\Tr(-)$ denotes Auslander's transpose (see \cite{AB2} for its definition). Then the natural map
	\[
	f : ((M^\ast)^\ast) \otimes_R (M^\ast) \longrightarrow \Hom_R((M^\ast)^\ast, (M^\ast)^\ast)
	\]
	is surjective, because the cokernel of $f$ is $\Ext_R^2(\Tr(M^\ast), (M^\ast)^\ast) = 0$ (see \cite[Proposition (2.6)(a)]{AB2}). It follows that $M^\ast$ is free as an $A$-module. Since $M$ is reflexive, $M \cong M^{\ast\ast}$ is also free. Thus $\id(A)$ is finite, which means that $A = R_{\fp}$ is Gorenstein. Hence $R$ satisfies $(G_1)$.
	
	Following Bass's conjecture, $R$ is Cohen--Macaulay. It remains to note that a ring satisfying $(S_2)$ and $(G_1)$ is quasi-normal.
\end{proof}

\begin{corollary}	
	Suppose $R$ is a quotient of a complete regular ring $S$ by an ideal $I$, and assume the injective dimension of the normal bundle $N$ is finite. Then $R$ is quasi-normal.
\end{corollary}

\begin{proof}
	This is similar to Proposition \ref{511}.
\end{proof}
\begin{corollary}	
Suppose $R$ is a quotient of a complete regular ring $S$ by a radical ideal $I$, and assume the injective dimension of the symbolic normal bundle $(I/I^{(2)})^\ast$ is finite. Then $R$ is quasi-normal.
\end{corollary}

\begin{proof}
This is similar to {Corollary}  \ref{511a}.
\end{proof}

\begin{corollary}
	Let $M$ be reflexive such that $\id_R(M) < \infty$ and $\mu(M) \leq 2$. Then $R$ is Gorenstein.
\end{corollary}

\begin{proof}
	Without loss of generality, we may assume that $(R,\fm)$ is complete in the $\fm$-adic topology. By the above proposition, the ring is quasi-normal. Now apply Proposition \ref{pqn}.
\end{proof}

\begin{corollary}
	Let $(R,\fm)$ be of type at most two, and let $M$ be an indecomposable reflexive module such that $\id_R(M) < \infty$ and $\mu(M) \leq 4$. Then $R$ is Gorenstein.
\end{corollary}

\begin{corollary}\label{cor7.7}
	Let $(R,\fm)$ be of type two, and let $M$ be a module such that $\id_R(M) < \infty$. Then $\mu(M)$ is even.
\end{corollary}

An $R$-module $M$ is called $k$-torsionless if $\Ext_R^i(\Tr M, R) = 0$ for all $i \leq k$. Recall that torsionless means $1$-torsionless and reflexive means $2$-torsionless. 

\begin{question}
	Let $k \geq 0$. Suppose there exists a nonzero $(k+1)$-torsionless module of finite injective dimension. When does $R$ satisfy property $(G_k)$?
\end{question}

\begin{proposition}\label{remf}
	Let $A$ be a homomorphic image of a Gorenstein ring. Suppose a nonzero module $N$ is $(k+2)$-torsionless and has finite injective dimension. Then $A$ satisfies property $(G_k)$.
\end{proposition}

\begin{proof}
	Let $\fp \in \Spec(A)$ be of height at most $k$. Set $R := A_{\fp}$ and $M := N_{\fp}$. The problem reduces to showing that $R$ is Gorenstein. The ring has a canonical module. This helps us apply Lemma \ref{sfs}, and thus $M$ is nonzero. Then, without loss of generality, we may assume in addition that $k = d := \dim R$, since $\ell$-torsionless behaves well with respect to localization. By definition, $M$ is reflexive and $\Ext_R^i(M^\ast, R) = 0$ for all $i \leq d$, because $M^\ast$ is the second syzygy of $\Tr M$. Thanks to Bass, $d = \dim R = \depth(R)$. We proceed by induction on $d$ to deduce that $M$ is free. The base case $d = 0$ is clear from Lemma \ref{in}. Using the induction hypothesis, we may assume in addition that $M$ is locally free over the punctured spectrum. Let $L := M^\ast$. Then $L$ is locally free over the punctured spectrum and $\id_R(L^\ast) \leq d$. Hence $\Ext_R^{d+1}(\Tr L, L^\ast) = 0$ and $\Ext_R^i(L, R) = 0$ for all $i \leq d$. Under these assumptions, Kimura proved in \cite[Theorem 2.4]{kim} that $L$ is free. Then $M \cong M^{\ast\ast} \cong L^\ast$ is also free. Thus $\id(R) \leq d$. Consequently, $R$ is Gorenstein.
\end{proof}

The following is our second main result of this section.

\begin{corollary}
	Suppose a nonzero module $M$ is $(d+2)$-torsionless and has finite injective dimension, where $d = \dim R$. Then $R$ is Gorenstein.
\end{corollary}

\begin{proof}
	We may assume the ring is complete. Then the result follows from Proposition \ref{remf}.
\end{proof}

One can prove the case $d = 0$ even without using Proposition \ref{remf}:

\begin{fact}
	Let $R$ be artinian. Suppose there exists a nonzero reflexive module $M$ of finite injective dimension. Then $R$ is Gorenstein.
\end{fact}

\begin{proof}
	We dualize $R^n \to M^\ast \to 0$. Then by definition, $M \cong M^{\ast\ast} \oplus X \cong R^n$ for some $n > 0$ and some $X$. Now, by the Krull--Remak--Schmidt theorem, $R$ is injective and therefore $R$ is Gorenstein.
\end{proof}
\medskip

\section{On the injective dimension of Hom}

This section is about:

\begin{question}
	(See \cite[Question 2.18]{q1}). If there exists a nonzero $R$-module $M$ such that $\Hom_R(M,M)$ has finite injective dimension, then is $R$ Gorenstein?
\end{question}

\begin{question}
	(See \cite[Question 4.9]{q2}). Let $M$ be a nonzero $R$-module such that the injective dimensions of $M$ and $\Ext^i_R(M,M)$ are finite for all $0 \leq i \leq \dim(R) - \depth(M)$. Is $R$ Gorenstein?
\end{question}

\subsection{Injective dimension of Hom}

We start with the following easy exercise:

\begin{fact}
	Question 8.1 is true when $R$ is artinian.
\end{fact}

\begin{proof}
	We may assume that $M$ is indecomposable. Apply $\Hom_R(-, M)$ to $0 \to \Syz(M) \to R^n \to M \to 0$ and obtain $0 \to \Hom_R(M, M) \to \Hom_R(R^n, M) \to C \to 0$, where $C$ is the cokernel. Since $\dim R = 0$, $\Hom_R(M, M)$ is injective; i.e., $M^n = \Hom_R(M, M) \oplus C$. By the Krull--Remak--Schmidt theorem, $M^m = \Hom_R(M, M)$ for some $m \leq n$. Hence $M = E_R(k)$, and consequently $R = \Hom_R(E_R(k), E_R(k))$ is injective. Therefore $R$ is Gorenstein.
\end{proof}

Here we present a modern proof of the above fact:

\begin{proof}
	We have $\Ext^1_R(M, \Hom_R(M, M)) = 0$, and so $M$ is free (see \cite{ta}). This gives the injectivity of $R$. Hence $R$ is Gorenstein.
\end{proof}

\begin{proposition}
	Let $M$ be a nonzero finite-length $R$-module such that the injective dimensions of $M$ and $\Hom_R(M, M)$ are finite. Then $R$ is Gorenstein.
\end{proposition}

\begin{proof}
	The ring is Cohen--Macaulay. Without loss of generality, we may assume the ring is complete. Then we use Matlis duality. It turns out that $\pd(\Hom_R(M, M)^\vee) < \infty$. But $\Hom_R(M, M)^\vee \cong M \otimes M^\vee$. In view of \cite[A.9(ii)]{celp}, $R$ is Gorenstein.
\end{proof}

\begin{proposition}
	Let $M \subseteq E_R(k)$ be a nonzero finite-length $R$-module such that the injective dimension of $\Hom_R(M, M)$ is finite. Then $R$ is Gorenstein and $\id(M) < \infty$.
\end{proposition}

\begin{proof}
	The ring is Cohen--Macaulay. Without loss of generality, we may assume the ring is complete. Dualizing the embedding $0 \to M \hookrightarrow E_R(k)$ yields the surjection $R \cong \widehat{R} \twoheadrightarrow M^\vee \to 0$. Let $I := \ker(\rho^\vee)$. Since $M^\vee = R/I$, we deduce that $I = \Ann(M^\vee) = \Ann(M)$. Recall from $\Hom_R(M, M)^\vee \cong M \otimes M^\vee$ that the module
	\[
	M = \frac{M}{IM} \cong M \otimes M^\vee = \Hom_R(M, M)^\vee
	\]
	has finite projective dimension. Thus $M^\vee = R/I$ has finite injective dimension. Since it is cyclic, we apply the mentioned result of Peskine--Szpiro to deduce that $R$ is Gorenstein. Because we have observed that $\pd(M) < \infty$, we can now deduce from the Gorenstein property that $\id(M) < \infty$.
\end{proof}

\begin{corollary}
	Let $M \subseteq E_R(k)$ be a nonzero finite-length $R$-module such that the projective dimension of $M$ is finite. Then $R$ is Gorenstein and $\id(\Hom_R(M, M)) < \infty$.
\end{corollary}

The following was proved in \cite{q2} by a spectral sequence argument:

\begin{proposition}
	Let $N$ be torsion-free with $\depth(N) \geq \dim R - 1$ such that the injective dimensions of both $N$ and $\Ext^\ast_R(M, N)$ are finite. Then $\pd(M) < \infty$.
\end{proposition}

\begin{proof}
	Without loss of generality, we may assume $\depth(N) = \dim R - 1$. By passing to syzygies, we may assume that $M$ is maximal Cohen--Macaulay. We also assume in addition that $R$ is generically Gorenstein. Consider the exact sequence
	\[
	0 \longrightarrow \omega_R^n \longrightarrow \omega_R^m \longrightarrow N \longrightarrow 0 \qquad (\ast)
	\]
	and recall that $\Ext^1_R(M, \omega_R) = 0$. This gives
	\[
	0 \to \Hom_R(M, \omega_R^n) \to \Hom_R(M, \omega_R^m) \to \Hom_R(M, N) \to \Ext^1_R(M, \omega_R) = 0.
	\]
	Since $\id(\Hom_R(M, N)) < \infty$, we deduce that
	\[
	\Ext^{\gg 0}_R(-, \Hom_R(M, \omega_R))^n \stackrel{\cong}{\longrightarrow} \Ext^{\gg 0}_R(-, \Hom_R(M, \omega_R))^m.
	\]
	If $\Ext^{\gg 0}_R(-, \Hom_R(M, \omega_R)) \neq 0$, then it follows that $n = m$, but this contradicts the rank consideration in $(\ast)$. Thus $\Ext^{\gg 0}_R(-, \Hom_R(M, \omega_R)) = 0$. By definition, $\id(\Hom_R(M, \omega_R)) < \infty$. Since $M$ is maximal Cohen--Macaulay, $\Hom_R(M, \omega_R)$ is also maximal Cohen--Macaulay. Hence $\Hom_R(M, \omega_R) \cong \bigoplus \omega_R$. Let $(-)^\dagger = D(-)$ be as before. Then
	\[
	M = M^{\dagger\dagger} \cong \left(\bigoplus \omega_R\right)^\dagger \cong \bigoplus R,
	\]
	as claimed.
\end{proof}

The ideal case has an affirmative answer:

\begin{proposition}\label{88}
	Let $I$ be a nonzero ideal of $R$ such that the injective dimension of $\Hom_R(I, I)$ is finite. Then $R$ is Gorenstein.
\end{proposition}

\begin{proof}
	Let $d = \dim R$. We proceed by induction on $d$. The case $d = 0$ is trivial. Suppose $d = 1$. In this case we are able to apply \cite[2.3]{q1} to deduce that $I = \Gamma_{\fm}(I) \oplus R^m$ for some $m \geq 0$. The ring is Cohen--Macaulay and $\depth(R) = d > 0$. Thus $\Gamma_{\fm}(R) = 0$. But $\Gamma_{\fm}(I) \subseteq \Gamma_{\fm}(R) = 0$. Hence $I \cong R^m$, and by rank considerations, $I \cong R$. Consequently, the injective dimension of $R \cong \Hom_R(I, I)$ is finite. Therefore $R$ is Gorenstein.
	
	Thus we may assume $d > 1$ and apply the inductive hypothesis to deduce that the natural map $\pi : R \to \Hom_R(I, I)$ is an isomorphism over the punctured spectrum. Recall that $\pi$ is injective because $I$ is torsion-free. From the exact sequence
	\[
	0 \longrightarrow R \longrightarrow \Hom_R(I, I) \longrightarrow C := \operatorname{Coker}(\pi) \longrightarrow 0,
	\]
	we deduce that $\Supp(C) \subseteq \{\fm\}$. Suppose $C$ is nonzero. Then
	\[
	\inf\{i : \Ext^i_R(C, R) \neq 0\} = \depth(R) = d > 1.
	\]
	This implies that the extension class of the sequence lies in $\Ext^1_R(C, R) = 0$, and so the sequence splits; i.e., $\Hom_R(I, I) \cong R \oplus C$. Since $C \subseteq \Gamma_{\fm}(\Hom_R(I, I)) = 0$, we see $C = 0$. Then $\Hom_R(I, I) \cong R$. Consequently, the injective dimension of $R$ is finite. Thus $R$ is Gorenstein.
\end{proof}

\begin{proposition}
	Let $\mathbb{Q} \subseteq R$ and $M$ be torsion-free such that the injective dimension of $\Hom_R(M, M)$ is finite. Then $R$ is Gorenstein.
\end{proposition}

\begin{proof}
	Since $\Supp(M) = \Spec(R)$, and by the previous argument, $M$ is locally free in codimension one. Recall that $R$ is Cohen--Macaulay, and hence satisfies $(S_2)$. Therefore, according to \cite[5.5]{Ki}, $M$ has constant rank. Since $\mathbb{Q} \subseteq R$, the rank is a unit in $R$. By \cite[3.1]{q1}, $R$ is Gorenstein.
\end{proof}

\begin{remark}
	What about big Cohen--Macaulay modules of finite injective dimension?
\end{remark}

By $\overline{R}$ we denote the integral closure of $R$ in the total ring of fractions of $R$. We learned the following through a joint work with Mahdavi \cite[6.15]{Elham}:

\begin{proposition}\label{811}
	Assume $R$ is analytically unramified. If $\id_R(\overline{R}) < \infty$, then $R$ is normal. In particular, $R$ is Gorenstein.
\end{proposition}

\begin{proof}
	By Bass'  conjecture, $R$ is Cohen--Macaulay and in particular satisfies Serre's $(S_2)$ condition. Let $\mathfrak{p} \in \Spec(R)$. Since being analytically unramified is a local property, $A := R_{\mathfrak{p}}$ is analytically unramified. Suppose $\operatorname{ht}(\mathfrak{p}) = 1$. Then $A \to \overline{A}$ is integral. Since $\id_R(\overline{R}) < \infty$, we obtain $\id_A(\overline{A}) < \infty$. But $\overline{A}$ is maximal Cohen--Macaulay, so $\overline{A} \cong \bigoplus \omega_A$. We may assume the direct sum is a singleton. Then
	\[
	A \subseteq \operatorname{Hom}_{\overline{A}}(\overline{A}, \overline{A}) \subseteq \operatorname{Hom}_A(\overline{A}, \overline{A}) = \operatorname{Hom}_A(\omega_A, \omega_A) = A.
	\]
	Hence $A = \operatorname{Hom}_{\overline{A}}(\overline{A}, \overline{A}) = \overline{A}$, so $A$ is regular. In particular, $R$ satisfies Serre's $(R_1)$ condition. Moreover,
	\[
	(S_2) + (R_1) \Longrightarrow R \text{ is normal}.
	\]
	Thus $R = \overline{R}$, and so $\id(R) = \id_R(\overline{R}) < \infty$. This means $R$ is Gorenstein. The proof is now complete.
\end{proof}

\begin{remark}
	Here we present another argument for Proposition \ref{811}. Indeed, let $C$ be the conductor ideal. Recall that it is a common ideal of $R$ and $\overline{R}$. Then $\overline{R} \subseteq \Hom_{\overline{R}}(C, C) \subseteq \Hom_R(C, C) \subseteq \overline{R}$. Consequently, $\id_R(\Hom_R(C, C)) = \id_R(\overline{R}) < \infty$. Now use Proposition \ref{88} to see that $R$ is Gorenstein. Thus $\pd_R(\overline{R}) < \infty$. Now apply \cite[6.15]{Elham} to see that the ring is normal as well.
\end{remark}

\subsection{Remarks on the prime characteristic case}

Rings in this subsection are of prime characteristic $p$, and we have the Frobenius map $F : R \to R$. Each iteration $F_n$ of $F$ defines a new $R$-module structure on the set $R$, and this $R$-module is denoted by $\up{\F_n} R$, where $a \cdot b = a^{p^n} b$ for $a, b \in R$. Recall that $R$ is said to be $F$-finite if for some (or equivalently, all) $n$, the algebra ${}^F_n R$ is finitely generated as an $R$-module.

\begin{observation}
	Let $R$ be $F$-finite and $F$-pure. If $\id_R(\Hom_R(\up{\F_n} R, \up{\F_n} R)) < \infty$, then $R$ is regular.
\end{observation}

\begin{proof}
	Since $R$ is $F$-pure, $R$ is a direct summand of $\up{\F_n} R$. This shows that $\up{\F_n} R$ is a direct summand of $\Hom_R(\up{\F_n} R, \up{\F_n} R)$. This, in turn, implies that $\id(\up{\F_n}R) < \infty$. According to \cite[Proposition 3.2]{q1}, $R$ is Gorenstein. Since the ring is Gorenstein, we have $\pd(\up{\F_n} R) < \infty$. By a result of Kunz, $R$ is regular.
\end{proof}

\begin{observation}
	Let $R$ be $F$-finite and $F$-injective with quotient singularity. Then $R$ is regular, provided $\id_R(\Hom_R(\up{\F_n} R, \up{\F_n} R)) < \infty$.
\end{observation}

\begin{proof}
	Since $R$ has quotient singularities, $\Hom_R(\up{\F_n} R, \up{\F_n}R)$ is maximal Cohen--Macaulay. In view of \cite[Proposition 3.2]{q1}, the ring is Gorenstein. But Gorenstein and $F$-injective imply $F$-pure. Now use the previous observation.
\end{proof}

\begin{corollary}
	Let $R$ be a $2$-dimensional $F$-finite and $F$-injective ring. Then $R$ is regular, provided $\id_R(\Hom_R(\up{\F_n} R, \up{\F_n}R)) < \infty$.
\end{corollary}

\begin{proof}
	Since $\depth(\Hom_R(\up{\F_n}R, \up{\F_n} R)) \geq \min\{2, \depth(R)\}$, it is maximal Cohen--Macaulay. It remains to apply the previous argument.
\end{proof}

There are some situations in which $R$ is a summand of $\up{\F_n} R$, and so the above observations can be applied to them as well.

\section{On the injective dimension of tensor products}

Recall the following question raised in \cite[Question 4.2]{celgo}:

\begin{question}
	Suppose $R$ is a $d$-dimensional local Cohen--Macaulay ring of minimal multiplicity with canonical module. Let $M$ be a maximal Cohen--Macaulay module such that $M \otimes M^\dagger \cong \omega_R$. Then is $M \cong R$ or $M \cong \omega_R$?
\end{question}

The following result removes the minimal multiplicity assumption from \cite[Corollary 4.5]{celgo} in the one-dimensional case.

\begin{proposition}
	Suppose $R$ is a $1$-dimensional local Cohen--Macaulay ring with canonical module. Let $I$ be a reflexive ideal such that $I \otimes I^\dagger \cong \omega_R$. Then $I \cong R$.
\end{proposition}

\begin{proof}
	Recall that $\depth(I) > 0$; otherwise $R/\fm \subseteq I \subseteq R$, which is impossible because $\depth(R) = 1$. By $\dagger$-duality we mean the dual with respect to $\Hom_R(-, \omega_R)$. Let $x$ be an $R$-regular element (which is then also $I$-regular). It is easy to see that
	\[
	E_{\overline{R}}(k) = \omega_{\overline{R}} = \overline{I} \otimes \overline{I}^\dagger,
	\]
	where $\overline{R}=R/xR$ and $\overline{I}=I/xI$. Applying duality again gives
	\[
	\overline{R} = \Hom(E_{\overline{R}}(k), E_{\overline{R}}(k)) = \Hom(\overline{I} \otimes \overline{I}^\dagger, \omega_{\overline{R}})
	= \Hom(\overline{I}, \overline{I}^{\dagger\dagger}) = \Hom(\overline{I}, \overline{I}).
	\]
	Since $I$ is reflexive, it is a second syzygy, and therefore $\overline{I} = I/xI \hookrightarrow \bigoplus \overline{R}$. In particular, $\overline{I}$ is torsionless, and hence
	\[
	0 = \Ext^1_{\overline{R}}(\Tr(\overline{I}), \overline{R}) = \Ext^1_{\overline{R}}\bigl(\Tr(\overline{I}), \Hom(\overline{I}, \overline{I})\bigr).
	\]
	Recall that $\depth(\overline{I}) = 0$. Then, thanks to \cite[Proposition 3.3(2)]{ta}, this implies that $\overline{I}$ is free over $\overline{R}$. Consequently, $I$ is free over $R$, so $I \cong R$.
\end{proof}

We now address Question 9.1 in the following  generality.

\begin{proposition}
	Suppose $R$ is a $d$-dimensional local Cohen--Macaulay ring with canonical module. Let $M$ be a totally reflexive module such that $M \otimes M^\dagger \cong \omega_R$. Then $M \cong R$.
\end{proposition}

\begin{proof}
	By the Auslander--Bridger formula, $\depth(M) = \dim R$. The desired statement behaves well with respect to reduction by regular sequences. Thus, by standard reductions, the claim reduces to the $1$-dimensional case. Applying the argument of the previous proposition, we conclude that $M \cong R$.
\end{proof}

\section{Quotients of finite injective dimension}

This section is about the following question:

\begin{question}
	(See \cite[Question 1.7]{q3}). If a finite direct sum of syzygy modules of $k := R/\fm$ surjects onto a nonzero $R$-module of finite injective dimension, then is $R$ regular?
\end{question}

The authors of \cite{q3} proved:

\begin{fact}\label{ii}
	(See \cite[Corollary 3.4]{q3}). If a finite direct sum of syzygy modules of $k$ surjects onto a nonzero maximal Cohen--Macaulay $R$-module $L$ of finite injective dimension, then $R$ is regular.
\end{fact}

Here is our main result of this section.

\begin{theorem}
	Let $d := \dim R$ and suppose $\Lambda \subseteq [d, \infty)$. If $\bigoplus_{j \in \Lambda} \operatorname{Syz}_j(k)$ surjects onto a nonzero $R$-module of finite injective dimension, then $R$ is regular.
\end{theorem}

\begin{proof}We may assume $\Lambda \subseteq [d, \infty)$ is finite.
In fact, without loss of generality, we may assume $\Lambda$ is a singleton. The ring $R$ is Cohen--Macaulay. Without loss of generality, we may also assume $R$ is complete. In particular, $R$ is equipped with a canonical module $\omega_R$. Let $N$ be the module of finite injective dimension that is a quotient of $\operatorname{Syz}_j(k)$, and consider its $\omega_R$-resolution
	\[
	0 \longrightarrow \omega_R^{r_p} \stackrel{\Phi_p}{\longrightarrow} \cdots \stackrel{\Phi_1}{\longrightarrow} \omega_R^{r_1} \stackrel{\varphi}{\longrightarrow} N \longrightarrow 0
	\]
	with the property that $\operatorname{im}(\Phi_i) \subseteq \fm \omega_R^{r_{i-1}}$ (see \cite[3.3.28(c)]{BH}). In particular, there exists an exact sequence
	\[
	0 \to \Theta := \ker(\varphi) \to \omega_R^n \stackrel{\varphi}{\longrightarrow} N \to 0
	\]
	with the convenient property that
	\[
	\Theta \subseteq \fm \omega_R^n \qquad (+)
	\]
	Let $\psi : \operatorname{Syz}_j(k) \twoheadrightarrow N$ and define $K := \ker(\psi)$; i.e., we have the exact sequence
	\[
	0 \to K \to \operatorname{Syz}_j(k) \stackrel{\psi}{\longrightarrow} N \to 0.
	\]
	From the first short exact sequence we deduce that $\Theta$ has finite injective dimension. Also, $\operatorname{Syz}_j(k)$ is maximal Cohen--Macaulay. These facts enable us to use Ischebeck's formula and deduce that
	\[
	\sup\{i : \Ext^i_R(\operatorname{Syz}_j(k), \Theta) \neq 0\} = \depth(R) - \depth(\operatorname{Syz}_j(k)) = d - d = 0.
	\]
	Thus
	\[
	\Ext^1_R(\operatorname{Syz}_j(k), \Theta) = 0 \qquad (\ast)
	\]
	Now we use the following well-known diagram:
	\[
	\begin{CD}
	@. @. 0 @. 0 \\
	@. @. @AAA @AAA \\
	0 @>>> \Theta @>f_1>> \omega_R^n @>>> N @>>> 0 \\
	@. @| @AfAA @AAA \\
	0 @>>> \Theta @>>> L @>>> \operatorname{Syz}_j(k) @>>> 0 \\
	@. @. @AAA @AAA \\
	@. @. K @= K @>>> 0 \\
	@. @. @AAA @AAA \\
	@. @. 0 @. 0
	\end{CD}
	\]
	\\
	
	From $(\ast)$ we deduce that the middle row splits, and so $L = \operatorname{Syz}_j(k) \oplus \Theta$. Applying $f$ to this yields
	\[
	\begin{array}{rcl}
	\omega_R^n &=& f(L) = f(\operatorname{Syz}_j(k) \oplus \Theta) \\
	&=& f(0 \oplus \Theta) + f(\operatorname{Syz}_j(k) \oplus 0) \\
	&=& f_1(\Theta) + f(\operatorname{Syz}_j(k) \oplus 0) \\
	&\stackrel{(+)}{\subseteq}& \fm \omega_R^n + f(\operatorname{Syz}_j(k) \oplus 0) \\
	&\subseteq& \omega_R^n.
	\end{array}
	\]
	In particular, $\omega_R^n = \fm \omega_R^n + f(\operatorname{Syz}_j(k) \oplus 0)$, and by Nakayama's lemma, $f(\operatorname{Syz}_j(k) \oplus 0) = \omega_R^n$. In other words, there exists a surjection
	\[
	\operatorname{Syz}_j(k) \cong \operatorname{Syz}_j(k) \oplus 0 \stackrel{f}{\twoheadrightarrow} \omega_R^n \twoheadrightarrow \omega_R.
	\]
	By Fact \ref{ii}, $R$ is regular.
\end{proof}

Rings in the rest of this section are Cohen--Macaulay.

\begin{remark}\label{104}
	Let $P$ be a prime ideal and suppose that there exists a surjection $P \twoheadrightarrow \omega_R$. Then the following holds for any $Q \in \Spec(R)$:
	\begin{itemize}
		\item[(i)] If $Q \subseteq P$, then $R_Q$ is regular.
		\item[(ii)] If $P \nsubseteq Q$, then $R_Q$ is Gorenstein.
	\end{itemize}
\end{remark}

\begin{proof}
	(i) From the surjection $P \twoheadrightarrow \omega_R$ we obtain another surjection
	\[
	\operatorname{Syz}(R_P / P R_P) = P R_P \twoheadrightarrow (\omega_R)_P = \omega_{R_P}.
	\]
	Thus $R_P$ is regular. Since localization of a regular ring is again regular, $R_Q$ is regular.
	
	(ii) Since $P \nsubseteq Q$, we have $P R_Q = R_Q$. From the surjection $P \twoheadrightarrow \omega_R$ we obtain another surjection
	\[
	R_Q = P R_Q \twoheadrightarrow (\omega_R)_Q = \omega_{R_Q}.
	\]
	In particular, $\omega_{R_Q}$ is cyclic, and therefore $R_Q$ is Gorenstein.
\end{proof}

\begin{corollary}\label{105}
	Let $R$ be a $2$-dimensional Cohen--Macaulay ring and let $P$ be a prime ideal such that there exists a surjection $P \twoheadrightarrow \omega_R$. Then $R$ has an isolated Gorenstein singularity.
\end{corollary}

\begin{proof}
	First we observe that $\operatorname{ht}(P) > 0$. If not, then $P \in \min(R) = \Ass(R)$. Thus $P = (0 : x)$ for some $x \in R$. It is clear that $x \neq 0$ because $P$ is proper. Then $xP = 0$, and from $P \twoheadrightarrow \omega_R$ we deduce that $x \omega_R = 0$. This contradicts the fact that $\omega_R$ is faithful. Hence $\operatorname{ht}(P) > 0$. The claim is clear if $\operatorname{ht}(P) > 1$ (see Fact \ref{ii}). Thus, without loss of generality, we may assume $\operatorname{ht}(P) = 1$.
	
	Let $Q \in \Spec(R)^\circ := \Spec(R) \setminus \{\fm\}$. Since $\dim(R) = 2$, one of the following holds:
	\begin{itemize}
		\item[(i)] $Q \subseteq P$, or
		\item[(ii)] $P \nsubseteq Q$.
	\end{itemize}
	By Remark \ref{104}, $R_Q$ is Gorenstein in either case. By definition, $R$ has an isolated Gorenstein singularity.
\end{proof}

\begin{corollary}\label{106}
	Let $R$ be a Cohen--Macaulay ring and let $P$ be a prime ideal such that there exists a surjection $P \twoheadrightarrow \omega_R$. Then $R$ is generically Gorenstein.
\end{corollary}

\begin{proof}
	Without loss of generality, we may assume $\dim(R) > 0$ (see Fact \ref{ii}). Let $Q \in \Ass(R) = \min(R)$. Also, recall from the previous argument that $\operatorname{ht}(P) > 0$. Then $P \nsubseteq Q$. Thanks to Remark \ref{104}(ii), $R_Q$ is Gorenstein. Thus $R$ is generically Gorenstein.
\end{proof}

\section{More on G-dimension}
\subsection{When is $\operatorname{Gdim}((\omega_R^\ast)^\dagger) < \infty$?}

In this section, $R$ is a Cohen--Macaulay ring with canonical module.
Recall that the notation $M^\ast$ (resp. $M^\dagger$) stands for $\Hom_R(M, R)$ (resp. $\Hom_R(M, \omega_R)$).

\begin{observation}
	Let $(S, \fn)$ be a regular local ring and $R := S / \fn^n$. If $\operatorname{Gdim}((\omega_R^\ast)^\dagger) < \infty$, then $R$ is a hypersurface.
\end{observation}

\begin{proof}
	We may assume $n > 1$, because otherwise $R$ is a field. Set $E := \omega_R$. We use the more common Matlis duality notation $(-)^{\vee}$ instead of $(-)^\dagger$. We know $\mathfrak{m} E^* = 0$. Thus
	\[
	E^* = \bigoplus_X R/\mathfrak{m},
	\]
	where $X \neq \emptyset$. Recall that
	\[
	(E^*)^{\vee} = \left( \bigoplus_X R/\mathfrak{m} \right)^{\vee} = \bigoplus_X (R/\mathfrak{m})^{\vee} = \bigoplus_X R/\mathfrak{m}.
	\]
	Hence $\operatorname{Gdim}(R/\mathfrak{m}) = \operatorname{Gdim}((\omega_R^\ast)^\dagger) < \infty$. This shows that $R$ is Gorenstein. In particular, its socle is $1$-dimensional as a vector space. But this implies
	\[
	E := \operatorname{Soc}(R) = \mathfrak{m}^{n-1} \;\Longrightarrow\; \mu(\mathfrak{m}^{n-1}) = 1 \;\Longrightarrow\; \mathfrak{m} \text{ is principal}.
	\]
	This means that $R$ is a hypersurface.
\end{proof}

The above observation extends to a more general setting.

\begin{definition}[Ramras]
	We say that a local ring is \emph{eventually BNSI} if there exists $\ell \geq 1$ such that for every non-free module $M$, we have $\beta_i(M) > \beta_{i-1}(M)$ for all $i > \ell$. 	If $\ell=1$ we say  $R$ is \emph{BNSI}.
\end{definition}

For example, $R := k[[x, y, \ldots]] / \mathfrak{m}^t$ is eventually BNSI for any $t \in \mathbb{N}$.

\begin{proposition}
	Suppose $R$ is  BNSI and equipped with a canonical module. If $\operatorname{Gdim}((\omega_R^\ast)^\dagger) < \infty$, then $R$ is a hypersurface. In particular, $R$ is  principal ideal ring. 
\end{proposition}

\begin{proof}
	It is a classic result that $\operatorname{depth}(R) = 0$ because $R$ is   BNSI (see \cite[Page~299 Remark]{ram}). Thanks to the Auslander--Bridger formula and the assumption $\operatorname{Gdim}((\omega_R^\ast)^\dagger) < \infty$, we observe that $(\omega_R^\ast)^\dagger$ is totally reflexive, and in particular reflexive. Over  BNSI rings, any reflexive module is free. Hence $(\omega_R^\ast)^\dagger$ is free.
	
	Since $R$ is equipped with a canonical module, it is Cohen--Macaulay. Because $R$ is Cohen--Macaulay and $\operatorname{depth}(R) = 0$, we have $\dim(R) = \depth(R) = 0$. Consequently, $\omega_R$ is injective. This allows us to deduce that
	\[
	(\omega_R^\ast)^\dagger = \Hom_R(\Hom_R(\omega_R, R), \omega_R) \cong \omega_R \otimes_R \Hom_R(R, \omega_R) = \omega_R \otimes_R \omega_R.
	\]
	In summary, $\omega_R \otimes_R \omega_R$ is free. Now consider the natural short exact sequence
	\[
	0 \lo \operatorname{Syz}(\omega_R) \lo R^n \lo \omega_R \lo 0,
	\]
	and tensor it with $\omega_R$. This yields the short exact sequence
	\[
	0  \lo\ker(\pi)  \lo \omega_R^n \stackrel{\pi} \lo \omega_R \otimes_R \omega_R  \lo 0.
	\]
	But recall that $\omega_R \otimes_R \omega_R$ is free, which means that the above sequence splits. Thus
	\[
	\ker(\pi) \oplus (\omega_R \otimes_R \omega_R) \cong \omega_R^n.
	\]
	Recall again that $R$ is artinian. This allows us to use the Krull--Remak--Schmidt decomposition theorem and deduce that $\omega_R \cong R$. Hence $R$ is Gorenstein. By the work of Ramras \cite[Poposition 4.1]{ram}, $R$ is a hypersurface, as claimed.
\end{proof}

\begin{remark}
	Let $(R, \mathfrak{m})$ be a $1$-dimensional nearly Gorenstein and far-flung Gorenstein ring such that $\operatorname{Gdim}((\omega_R^\ast)^\dagger) < \infty$. We leave it to the reader to check that $R$ is Gorenstein.
\end{remark}

\subsection{G-dimension of conormal module}

Let $(T,\m)$ be any Noetherian local ring, and suppose that
$Q=(a_1,\dots,a_h)$ is generated by a $T$-regular sequence of
length $h\ge1$. Put $B=T/Q$, and for $q\ge2$ set
\[
A=T/Q^q,\qquad M=Q^q/Q^{2q},\qquad
F=\gr_Q(T)\cong B[X_1,\dots,X_h].
\]
The last isomorphism is the standard associated-graded property of
an ideal generated by a regular sequence.

\begin{lemma} \label{lem:depth}
	For every $q\ge2$ one has
	\[
	\pd_T A=\pd_T M=h,
	\qquad \depth_A M=\depth A=\depth T-h.
	\]
\end{lemma}
\begin{proof}
	For each $r\ge1$ the graded piece $Q^{r-1}/Q^r$ is a free
	$B$-module. Since $\pd_T B=h$, induction using
	\[
	0\longrightarrow Q^{r-1}/Q^r\longrightarrow
	T/Q^r\longrightarrow T/Q^{r-1}\longrightarrow0
	\]
	gives $\pd_T(T/Q^r)\le h$. Moreover, $T/Q^r$ has support
	$V(Q)$, and every prime containing $Q$ contains a $T$-regular
	sequence of length $h$. Thus
	$\operatorname{grade}_T(T/Q^r)\ge h$; hence
	$\pd_T(T/Q^r)=h$.
	The exact sequence
	\[
	0\longrightarrow M\longrightarrow T/Q^{2q}
	\longrightarrow T/Q^q\longrightarrow0
	\]
	shows $\pd_T M\le h$. The $Q$-adic filtration of $M$ has
	successive quotients $Q^j/Q^{j+1}$, $q\le j<2q$, each a free
	$B$-module. Therefore $\operatorname{grade}_T M\ge h$, so
	$\pd_T M=h$ as well. Apply the Auslander--Buchsbaum formula
	over $T$. Finally, depth computed with the maximal ideal of $T$
	agrees with depth computed over the respective quotient ring,
	because both modules are annihilated by the corresponding kernel.
\end{proof}

\begin{lemma} \label{lem:dual}
	If $h\ge2$ and $q\ge2$, then every $T$-linear map
	$f:Q^q\to T/Q^q$ has image in $Q^{q-1}/Q^q$. Consequently,
	there is a natural isomorphism of $B$-modules
\begin{enumerate}
	\item[$\mathrm{(i)}$]  $\Hom_A(M,A)
	\cong\Hom_B\bigl(Q^q/Q^{q+1},\; Q^{q-1}/Q^q\bigr)$,
	
	\item[$\mathrm{(ii)}$]
	In particular,
$Q\,\Hom_A(M,A)=0\quad\text{and}\quad Q M^{**}=0.$
\end{enumerate}	
\end{lemma}
\begin{proof}
	An $A$-linear homomorphism $M\to A$ is the same thing as a
	$T$-linear map $Q^q\to T/Q^q$: every such map vanishes on
	$Q^{2q}$ automatically. Let $f$ be such a map.
Suppose its image is not contained in $Q^{q-1}/Q^q$. Let $d\le q-2$
	be the least integer for which
	$f(Q^q)\subseteq Q^d/Q^q$. Passing to the degree-$d$
	initial terms gives a nonzero $B$-linear map
$$(*):\quad	\bar f:Q^q/Q^{q+1}\longrightarrow Q^d/Q^{d+1}.$$
	Write $\mathfrak a=(X_1,\dots,X_h)\subset F$. The standard
	generators of $\mathfrak a^q$ are the monomials of total degree
	$q$, and their first syzygies are generated by the exchange
	relations $X_i X^\alpha=X_j X^\beta$ whenever
	$X_i X^\alpha$ and $X_j X^\beta$ are the same monomial.
	The corresponding relations for the generators of $Q^q$ imply
$X_i\bar f(X^\alpha)=X_j\bar f(X^\beta)$ in
$F_{d+1}$
	since $d+1<q$. Thus (*) extends to a nonzero
	homogeneous $F$-linear map
	\[
(+):\quad	\mathfrak a^q\longrightarrow F
	\quad\text{of degree }d-q<0.
	\]
	But $X_1,X_2$ is an $F$-regular sequence in $\mathfrak a$.
	It follows from $\operatorname{grade}_F(\mathfrak a^q)\ge2$
	and the exact sequence
	$0\to\mathfrak a^q\to F\to F/\mathfrak a^q\to0$ that
$\Hom_F(\mathfrak a^q,F)\cong F$
	as graded modules: the maps are multiplication by elements of $F$.
	There are no maps of negative degree, contrary to (+).
	Therefore $f(Q^q)\subseteq Q^{q-1}/Q^q$.
	Since $Q$ kills this target, $f$ kills $Q^{q+1}$ and factors
	uniquely through the $B$-module $Q^q/Q^{q+1}$.
	Conversely, every $B$-linear map between the modules on the
	right of (i) defines a $T$-linear map $Q^q\to T/Q^q$.
	This proves (i). Multiplication by $Q$ is zero on the
	image of every $f$, hence $QM^*=0$, and the same holds for
	$M^{**}=\Hom_A(M^*,A)$.
\end{proof}

\begin{theorem}\label{thm:power}
	Let $(T,\m)$ be Noetherian local, let $Q$ be generated by a
	regular sequence of length $h\ge1$, and let $q\ge2$. Then
	\[
	{
		\Gdim_{T/Q^q}(Q^q/Q^{2q})=
		\begin{cases}
			0,& h=1,\\
			\infty,&h\ge2.
	\end{cases}}
	\]
	In the second case $Q^q/Q^{2q}$ is not reflexive over $T/Q^q$.
\end{theorem}
\begin{proof}
	For $h=1$, writing $Q=(a)$ with $a$ a nonzerodivisor,
	multiplication by $a^q$ identifies
	$Q^q/Q^{2q}$ with $T/(a^q)$, so the module is free.
	Suppose $h\ge2$. By Lemma~\ref{lem:dual}, $QM^{**}=0$.
	On the other hand,
$QM=Q^{q+1}/Q^{2q}\ne0,$
	since $q+1<2q$ and the associated graded ring of $Q$
	is a polynomial ring over $T/Q$. The canonical homomorphism
	$M\to M^{**}$ therefore annihilates the nonzero submodule
	$QM$ and cannot be an isomorphism.
	If $\Gdim_A M<\infty$, Lemma~\ref{lem:depth} and the
	Auslander--Bridger formula would give
	$\Gdim_A M=\depth A-\depth_A M=0$, forcing $M$ to be
	reflexive. This is a contradiction.
\end{proof}

\begin{corollary} \label{cor:prime}
	Let $(S,\n)$ be regular local, $0\ne P\subsetneq S$ a prime,
	and $q\ge2$. For $I=P^q$ \emph{or} $I=P^{(q)}$,
	\[
	\Gdim_{S/I}(I/I^2)<\infty
	\quad\Longleftrightarrow\quad\htop(P)=1.
	\]
	If $\htop(P)=1$, $I$ is principal and $S/I$ is a
	hypersurface. If $\htop(P)\ge2$, the indicated G-dimension
	is infinite.
\end{corollary}
\begin{proof}
	Localize at $P$. Both choices satisfy
	$IS_P=(PS_P)^q$. The local ring $S_P$ is regular of dimension
	$h=\htop(P)$, and $PS_P$ is generated by a regular system
	of parameters. Localization preserves finite G-dimension,
	so Theorem~\ref{thm:power} excludes $h\ge2$.
	For $h=1$, the UFD property of $S$ gives $P=(a)$, hence
	$P^q=P^{(q)}=(a^q)$, and $I/I^2\cong S/I$.
\end{proof}

\begin{corollary}\label{cor:radical}
	Let $S$ be a regular local ring, $0\ne J\subsetneq S$ a
	radical ideal, and $q\ge2$. Then
	\[
	\Gdim_{S/J^q}(J^q/J^{2q})<\infty
	\quad\Longleftrightarrow\quad J\text{ is principal}.
	\]
	In the finite case, $S/J^q$ is Gorenstein.
\end{corollary}
\begin{proof}
	For each minimal prime $P$ of $J$, radicality gives
	$JS_P=PS_P$. Theorem~\ref{thm:power}, after localization
	at $P$, forces $\htop P=1$ if the G-dimension is finite.
	A regular local ring is factorial. Since $J$ is radical,
	we may write $J=\bigcap_i P_i$ over its minimal primes;
	each $P_i=(f_i)$ for a pairwise nonassociated irreducible
	$f_i$. In a UFD this intersection is $(\prod_i f_i)$.
	The converse follows from the principal case of
	Theorem~\ref{thm:power}.
\end{proof}

\begin{corollary}
	\label{cor:unmixed}
	Let $S$ be regular local, and let $J$ be a nonzero proper
	\emph{unmixed} ideal such that $J_P$ is generated by a
	regular sequence for every minimal prime $P$ of $J$.
	For $q\ge2$,
	\[
	\Gdim_{S/J^q}(J^q/J^{2q})<\infty
	\quad\Longleftrightarrow\quad J\text{ is principal}.
	\]
\end{corollary}
\begin{proof}
	Theorem~\ref{thm:power}, applied to $J_PS_P$ at each
	minimal prime $P$, forces all such primes to have height one.
	An unmixed height-one ideal in a regular local UFD is principal:
	its primary components are contractions of powers of the
	height-one primes in the corresponding discrete valuation
	localizations. The intersection of those principal ideals is
	principal. The converse is immediate.
\end{proof}

\begin{acknowledgement}
	I thank Ghosh, Gheibi, and Puthenpurakal for their comments.
\end{acknowledgement}
\maketitle

\maketitle


\begin{thebibliography}{99}
	 \bibitem{Ao}
	Y. Aoyama,   Complete local ($S_{n-1}$) rings of type $n\geq 3$ are Cohen-Macaulay,  Proc. Japan Acad. Ser. A  {\bf{70}} (1994), no. 3, 80–83.
	
 \bibitem{finitsup}
 M. Asgharzadeh,
 \emph{Finite support of tensor products}, arXiv:1902.10509   [math.AC].
 
  \bibitem{moh2}
 M. Asgharzadeh, \emph{Reflexivity revisited}, arXiv:1812.00830  [math.AC].
 
\bibitem{moh3}
M. Asgharzadeh, \emph{A note on Cohen-Macaulay descent},  arXiv:2011.04525  [math.AC].
\bibitem{Elham}M. Asgharzadeh, E. Mahdavi,
 \emph{Remarks on modules of finite projective dimension}, arXiv:2602.09899 [math.AC].

\bibitem{ini}
M. Asgharzadeh, \emph{On the initial Betti numbers}, arXiv:1912.10864 [math.AC].

 \bibitem{comment} M. Auslander, \emph{Comments on the functor Ext}, Topology {\bf{8}} (1969), 151--166.
 
 \bibitem{baus} M. Auslander, \emph{
 	Remarks on a theorem of Bourbaki},
 	Nagoya Math. J. {\bf{27}} (1966), 361–369.
 	\bibitem{AB2}
 	M. Auslander and M. Bridger, \emph{Stable module theory}, Mem. of
 	the AMS  {\bf94}, Amer. Math. Soc., Providence 1969.
 \bibitem{au} M.
 Auslander,  \emph{Modules over unramified regular local rings}, Illinois J. Math. {\bf{5}} (1961) 631-647.
 
 \bibitem{b}
B.  Briggs, \emph{Vasconcelos' conjecture on the conormal module}, Invent. Math. {\bf{227}} (2022), no. 1, 415–428.

\bibitem{bass}
H. Bass, \emph{On the ubiquity of Gorenstein rings},
Math. Z. {\bf{82}} (1963), 8--28.



\bibitem{bass2}
H. Bass, \emph{Injective dimension in noetherian rings}, Trans. Amer. Math. Soc. {\bf{102}} (1963), 18-29. 

\bibitem{BH}{W. Bruns and J. Herzog}, {\it Cohen-Macaulay rings},  Cambridge Studies in Advanced Mathematics,
{\bf 39},
Cambridge University Press, Cambridge, 1993.

\bibitem{BE}
D.
Buchsbaum, D. Eisenbud,  \emph{Some structure theorems for finite free resolutions}, Advances in Math. {\bf12} (1974), 84–139.
 
 
 \bibitem{cel}
O.  Celikbas, U. Le, H. Matsui,  \emph{On the depth and reflexivity of tensor products}, J. Algebra  {\bf606} (2022), 916–932.
\bibitem{celgo}O. 
Celikbas, Sh. Goto, R. Takahashi, N. Taniguchi, {\it On the ideal case of a conjecture of Huneke and Wiegand}, Proc. Edinb. Math. Soc. (2) {\bf62} (2019), no. 3, 847–859

\bibitem{celp}
O. Celikbas, T. Kobayashi, S. Dey, {\it On the projective dimension of tensor products of modules,} arxiv.org/abs/2304.04490.


\bibitem{dey}S. Dey, D. Ghosh, {\it
Finite homological dimension of Hom, vanishing of Ext, and applications to divisor class group},
arXiv:2310.10607 [math.AC].


\bibitem{fo}
H. B. Foxby, {\it On the $\mu^i$ in a minimal injective resolution II}, Math. Scand.
{\bf41} (1977), 19-44.

   \bibitem{G}
M. Gheibi, R. Takahashi,
{\it   Characteristic modules over a local ring},
J.  Algebra,  to appear.

   \bibitem{q3}
D. Ghosh, A. Gupta, Tony J. Puthenpurakal,
{\it   Characterizations of regular local rings via syzygy modules of the residue field},
J. Commut. Algebra {\bf10} (2018), no. 3, 327–337.
\bibitem{q2} D. Ghosh, Tony J. Puthenpurakal,
{\it Gorenstein rings via homological dimensions, and symmetry in vanishing of Ext and Tate cohomology},  Algebr. Represent. Theory {\bf27} (2024), no. 1, 639–653.

\bibitem{q1}D. Ghosh, R. Takahashi, {\it
	Auslander-Reiten conjecture and finite injective dimension of Hom}, 
Kyoto J. Math. {\bf64} (2024), no. 1, 229--243.







\bibitem{GR}M. 	Raynaud and L. Gruson,  \emph{Critères  de platitude et de projectivité. Techniques de "platification'' d'un module}, Invent. Math.  {\bf{13}} (1971), 1–89.

\bibitem{nano}
R. Holanda, Cleto B. Miranda-Neto,
\emph{Vanishing of (co)homology, freeness criteria, and the Auslander-Reiten conjecture for Cohen-Macaulay Burch rings}, arXiv:2212.05521. 


 \bibitem{h}C. Huneke, \emph{
 	The primary components and integral closures of ideals in 3-dimensional regular local rings}, Math. Ann.  {\bf{275}} (1986), 617--635.


\bibitem{hh}D.
Hanes and C. Huneke, \emph{Some criteria for the Gorenstein property}, 
J. Pure Appl. Algebra {\bf{201}} (2005), no. 1-3, 4--16.

 \bibitem{ta} K. Kimura, Y. Otake and R. Takahashi, \emph{Maximal Cohen-Macaulay tensor products and vanishing of Ext modules},
Bull. Lond. Math. Soc.{\bf{ 54}} (2022), no. 6, 2456-2468.
\bibitem{Ki} K. Kimura, J. Lyle and A. J. Soto Levins, {\it
On the depth of tensor products over Cohen-Macaulay rings},  arXiv:2505.00441.

\bibitem{kim}K. Kimura, \emph{
Auslander-Reiten conjecture for normal rings}, arXiv:2304.03956 [math.AC], Israel J. Math
to appear. 

\bibitem{kunz}E. Kunz, \emph{Almost complete intersections are not Gorenstein rings}, J. Algebra {\bf{28}} (1974), 111–115.

\bibitem{mat}
H. Matsumura, \emph{Commutative ring theory}, Cambridge Studies in Advanced Math, \textbf{8}, (1986).

\bibitem
{E}{E. Matlis}, {\it 1-dimensional Cohen-Macaulay Rings}. Lecture Notes in Mathematics, Vol.  {\bf327}. Springer-Verlag, Berlin-New York, 1973.

\bibitem{matsuoka}
T. Matsuoka, \emph{On almost complete intersections}, Manuscripta Math. {\bf{21}} (1977),
329-340.


\bibitem{ram}
M. Ramras, \emph{Betti numbers and reflexive modules}, Ring theory (Proc. Conf., Park City, Utah, 1971),   297~-308. Academic Press, New York, 1972.

\bibitem{MILL}M.
Miller, \emph{Bourbaki's theorem and prime ideals},
J. Algebra {\bf64} (1980), no. 1, 29–36.
\bibitem{sim}A.M.
Simon, \emph{Some homological properties of complete modules}, Math. Proc. Camb. Phil. Soc.
{\bf108}, 231–246 (1990).


\bibitem{sha} R. Y.~Sharp, 
\emph{Finitely generated modules of finite injective dimension over certain Cohen-Macaulay rings},
Proc. London Math. Soc. (3) {\bf{25}} (1972), 303–328.





\bibitem{PS}
Ch. Peskine and L. Szpiro, \emph{Dimension projective finie et cohomologie locale},
Publ. Math. IHES.  {\bf42} (1973),  47--119.


\bibitem{int}
P. Roberts, \emph{Le th\'{e}or\`{e}me d' intersection}, C. R. Acad. Sci. Paris Ser. I Math.,
{\bf{304}}  (1987),  177-180.

\bibitem{int1}
P. Roberts, \emph{Multiplicities and Chern classes in local algebra}, Cambridge
Tracts in Mathematics, vol. {\bf{133}}, Cambridge University Press, Cambridge (1998).




  \bibitem{wol}W. V.
Vasconcelos, \emph{Quasi-normal rings},
Illinois J. Math. \textbf{14} (1970), 268--273.

\bibitem{v1}
W. V. Vasconcelos, \emph{Divisor theory in module categories}, North-Holland Mathematics Studies {\bf14} (North-Holland Publishing Co., Amsterdam, 1974).

\bibitem{v}
W. V. Vasconcelos,  \emph{ On the homology of $I/I^2$}, Comm. Algebra {\bf{6}} (1978), no. 17, 1801-1809.

\end{thebibliography}
\end{document}